\documentclass{amsart}
\usepackage{amssymb}
\usepackage[backref=page]{hyperref}
\DeclareMathOperator{\dist}{dist}
\DeclareMathOperator{\supp}{supp}
\numberwithin{equation}{section}
\newtheorem{theorem}{Theorem}[section]
\newtheorem{corollary}[theorem]{Corollary}

\newtheorem{proposition}[theorem]{Proposition}
\newtheorem{definition}[theorem]{Definition}

\title[The divergence set for the wave equation in higher dimensions]{The divergence set for the wave equation in higher dimensions}

\author{Xiumin Du}
\address{Northwestern University}
\email{xdu@northwestern.edu}

\author{Terence L.~J.~Harris}
\address{
The University of Queensland
}
\email{terry.harris@uq.edu.au}

\author{Jianhui Li}
\address{
University of Jyv\"askyl\"a
}
\email{jianhui.j.li@jyu.fi}

\begin{document} 
\begin{abstract} It is shown that if $u$ solves the wave equation in $\mathbb{R}^{4+1}$ with initial data $u(\cdot,0) = u_0(\cdot) \in H^s$ and $\partial_tu(\cdot,0) = u_1(\cdot ) \in H^{s-1}$, where $0.5 < s \leq 0.55$, then $u(x,t) \to u_0(x)$ and $\partial_tu(x,t) \to u_1(x)$ pointwise as $t \to 0$, for all $x$ outside an exceptional set of Hausdorff dimension at most $6-4s$. In a very small range of $s$, this verifies a conjecture of Barcel\'o, Bennett, Carbery, and Rogers. More generally,  a partial improvement to the exceptional set bound in $\mathbb{R}^{n+1}$ is obtained for $n \geq 4$ and $1/2 < s < n/4$. 
\end{abstract}
\maketitle

\section{Introduction}

Let $u$ be the solution to the wave equation in $\mathbb{R}^{n+1}$, 
\[ \Delta u - u_{tt} = 0, \]
with initial data $u_0(x) = u(x,0)$ and $u_1(x) = u_t(x,0)$. For a function $f$ on $\mathbb{R}^n$, and $s \in \mathbb{R}$, let
\[ \|f\|_{H^s} = \left( \int_{\mathbb{R}^n} \left\lvert \widehat{f}(\xi)\right\rvert^2 \left(1+|\xi|^2\right)^{s} \,d\xi\right)^{1/2}. \]It was conjectured in \cite[Conjecture 1.1]{hamkolee} (attributed to Barcel\'o, Bennett, Carbery, and Rogers \cite{BBCR}) that for $n \geq 3$, if $u_0 \in H^s$ and $u_1 \in H^{s-1}$ with $1/2 < s < n/2$, then $u(x,t) \to u_0(x)$ and $\partial_t u (x,t) \to u_1(x)$ for all $x$ outside a set of Hausdorff dimension at most 
\[ \begin{cases} n+2 - 4s & 1/2 < s \leq 1 \\
n-2s & 1 < s \leq n/2, \end{cases} \]
with this bound known to be the best possible (if true).
Equivalently, the conjecture is that
\begin{equation} \label{divergencesup} \mathcal{D}_n(s) = \begin{cases} n+2 - 4s & 1/2 < s \leq 1 \\
n-2s & 1 < s \leq n/2, \end{cases} \end{equation}
where $\mathcal{D}_n(s)$ is 
the supremum of the Hausdorff dimension of the set of $x$ for which $(u(x,t), \partial_t u(x,t)) \not \to (u_0(x), u_1(x))$ as $t \to 0$, over all $(u_0,u_1) \in H^s \times  H^{s-1}$. The question in \cite{BBCR,hamkolee} is posed in terms of the convergence of the wave operator $e^{it\sqrt{-\Delta}} f(x)$ to $f(x)$ as $t\to 0$ when $f \in H^s$, but the solution $u$ to the wave equation can be written as a linear combination of wave operators applied to $u_0$ and the multiplication of $u_1$ on the Fourier side by $|\xi|^{-1}$, so the two formulations are equivalent. 

To explain the range $s > 1/2$ in \eqref{divergencesup}, it is known that there are examples with $u_0 \in H^{1/2}$ and $u_1 = 0$, such that $u(x,t)$ does not converge to $u_0(x)$ as $t \to 0$ on a  positive Lebesgue measure set of $x \in \mathbb{R}^n$; Walther \cite{walther} showed that for any $n \geq 1$ the maximal inequality
\[ \left(\int_{\mathbb{R}^n} \sup_{0 < t < 1} \left\lvert e^{it\sqrt{-\Delta}} f(x) \right\rvert^2 \, dx\right)^{1/2} \leq C\|f\|_{H^{1/2}}, \] 
fails, and even the weak-type $H^{1/2} \to L^{2,\infty}$ version was shown to fail in \cite[Lemma~ A2]{hamkolee}, so by Stein's maximal principle this implies the failure of a.e.~pointwise convergence. Here
\[ e^{it\sqrt{-\Delta}} f(x) = \int_{\mathbb{R}^n} e^{2\pi i \left( \langle x, \xi \rangle + t |\xi| \right)} \widehat{f}(\xi) \, d\xi. \] Cowling \cite{cowling} showed that pointwise a.e.~convergence \emph{does} hold if $s> 1/2$. It is easy to see that pointwise convergence holds for \emph{every} $x$ if $s > n/2$, which explains the restriction $s \leq n/2$ in \eqref{divergencesup}. 

The sharp value of $\mathcal{D}_2(s)$ was found for all $s$ by Barcel\'o, Bennett, Carbery, and Rogers \cite{BBCR}, and the conjectured value of $\mathcal{D}_3(s)$ in \eqref{divergencesup} was confirmed for all $s$ by Ham, Ko, and Lee \cite{hamkolee}.  For $n \geq 4$, the previous best known bounds from \cite{hamkolee} can be summarised as
\begin{equation} \label{knownbound}  \begin{cases} n+2 - 4s & \frac{1}{2} < s \leq 1 \\
n-2s & 1 < s \leq \frac{n}{2} \end{cases}\leq \mathcal{D}_n(s)\leq \begin{cases} n - \frac{(2s-1)(n-1)}{n-2} & \frac{1}{2} < s < \frac{n}{4} \\
\frac{3n+1}{2} - 4s & \frac{n}{4} \leq s < \frac{n+1}{4} \\
n-2s  & \frac{n+1}{4} \leq s \leq \frac{n}{2}. \end{cases} \end{equation}
 Therefore, prior to this work the conjecture was open in the range $1/2 < s < (n+1)/4$ when $n \geq 4$. 
 
 The following theorem verifies the conjecture when $n=4$ in the very small range $1/2 < s \leq 11/20$, and improves the upper bound in \eqref{knownbound} for all $n \geq 4$ when $1/2 < s < n/4$, with the exception of the point $s=\frac{2n-3}{2(n-1)}$ where the two upper bounds match. 
\begin{theorem} \label{maintheorem} Let $n\geq 4$ and $1/2 < s <n/4$. Then
\[ \mathcal{D}_n(s) \leq
\begin{cases}
n - \frac{(2s-1)(n-2)}{n - 3} & \frac{1}{2} < s \leq \frac{4n-5}{2(3n-2)} \\
n - \frac{(2s-1)(2n)(n - 1) +1}{2n^2 - 4n + 1} &  \frac{4n-5}{2(3n-2)} < s \leq  \frac{2n-3}{2(n - 1)}  \\
n - \frac{(2s-1)n(n - 1)(n - 3) - 1}{n^3 - 5n^2 + 6n - 1} & \frac{2n-3}{2(n - 1)} < s \leq \frac{n^3-3n^2+4n-3}{4n(n-2)} \\
n - \frac{(2s-1)(2n) + 1}{2(n - 1)} & \frac{n^3-3n^2+4n-3}{4n(n-2)} < s < \frac{n}{4}.
\end{cases}
\]
When $n=4$, equality holds in the first part, i.e. $D_4(s) = 6-4s$ for $1/2 < s \leq 11/20$. 
\end{theorem}

 Using the Kolmogorov–Seliverstov–Plessner linearisation argument, Ham, Ko, and Lee \cite{hamkolee} showed that upper bounds for $\mathcal{D}_n(s)$ follow from weighted $L^2$ inequalities for the wave equation, which is how Theorem~\ref{maintheorem} is proved. The  reduction to a weighted $L^2$ inequality is explained in Section~\ref{weightreduction}. Some ideas in the proof of the weighted $L^2$ inequality are explained in the next subsection. 

\subsection{Comments on the proof} The initial setup to the proof of the weighted $L^2$ inequality here follows that of Du and Zhang \cite{duzhang} for the Schrödinger equation. There, they divide the integration domain into ``broad'' and ``narrow'' points. In the broad case they use the $n$-linear Fourier restriction inequality in $\mathbb{R}^n$, and in the narrow case they use induction on scales and decoupling for an $(n-2)$-dimensional paraboloid. 

Here, the dichotomy is done at the $(n-1)$-linear level in $\mathbb{R}^{n+1}$, so that the broad case uses a weighted $(n-1)$-linear cone restriction inequality, and  the narrow case uses decoupling on an $(n-3)$-dimensional cone. When $n=4$, a similar dichotomy occurs in \cite{ouwang}, where Ou and Wang use the (unweighted) trilinear restriction inequality for the cone in $\mathbb{R}^5$ to prove the sharp result for the cone Fourier restriction problem in $\mathbb{R}^5$. The weighted trilinear cone restriction inequality here is proved via polynomial partitioning, following \cite{guth,ouwang}.  The polynomial partitioning method separates into cellular, tangential, and transverse cases. The improvement over \cite{harris2018} comes from modifying the proof of the tangential case in a way that makes different use of the weight, which gives an inequality which is better than the narrow case for a range of dimensions of the weight. 

To be a bit more precise,  for the purposes of induction, the weighted $(n-1)$-broad cone restriction inequality is proved under the more general assumption that the input function concentrates on wave packets close to an $m$-dimensional algebraic variety in $\mathbb{R}^{n+1}$, when $m \in \{n-1,n,n+1\}$, and the proof uses induction on $m$. In the unweighted approach from \cite{guth,ouwang}, when $m \in \{n,n+1\}$ the tangential case interpolates between the lower dimensional case and a simple $L^2 \to L^2$ inequality from Plancherel. The approach in \cite{harris2018} used the weight by replacing this $L^2$ inequality with a weighted $L^2$ inequality. Here, when $n-1\leq \alpha < n$, instead of interpolating with an $L^2$ inequality,  Hölder's inequality is used to lift the exponent from the $m$-dimensional case to the exponent for the $(m-1)$-dimensional case. The loss from Hölder's inequality is captured by a bound on the number of unit balls satisfying a non-concentration condition which also lie inside a large neighbourhood of an $(m-1)$-dimensional algebraic variety. When $m=n$ the bound uses Wongkew's theorem \cite{wongkew}, but when $m=n+1$ the bound uses only the non-concentration condition. Wongkew's theorem was used in a similar way in \cite{duliwangzhang}. 

When $\alpha < n-1$, the proof replaces a couple of steps with (multilinear) refined decoupling, but since the sharp bound is not obtained in this range for any $\alpha$ or $n$,  the explanation of this part will be left to the proof. 

\subsection*{Acknowledgements} The first author is partially supported by the NSF grant DMS-2237349.

\section{Reduction to weighted \texorpdfstring{$L^2$}{L2} inequality} \label{weightreduction}

Assume that $n \geq 2$, and let 
\[ Ef(x,t) = \int_{\mathbb{R}^n} e^{2\pi i (\langle x, \xi \rangle + t|\xi|) } f(\xi) \, d\xi, \]
so that $E\widehat{f}(x,t) = e^{it\sqrt{-\Delta}} f(x)$. It is standard that an upper bound $\mathcal{D}_n(s) \leq \alpha$ would follow from an $L^2$ maximal inequality\footnote{It is well known that apart from $\epsilon$ losses in $s$ or $\alpha$, there is no difference in requiring an $L^2$ norm on the left-hand side of \eqref{fractal} instead of an $L^1$ norm, for probability measures $\mu$. Similarly, for a probability measure $\mu$, $c_{\alpha}(\mu)$ in \eqref{fractal} is interchangeable with the energy $I_{\alpha}(\mu)$ up to $\epsilon$ losses in $s$ or $\alpha$.}
 \begin{equation} \label{fractal} \left(\int \sup_{t \in [0,1]} | E \widehat{f} (x,t) |^2 \, d\mu(x)\right)^{1/2}  \lesssim \|f\|_{H^s} c_{\alpha}(\mu)^{1/2}, \end{equation}
for any Borel measure $\mu$ on $B_n(0,1)$, where $c_{\alpha}(\mu) = \sup_{x \in \mathbb{R}^n, r >0} \frac{ \mu(B(x,r))}{r^{\alpha}}$. To prove \eqref{fractal} for a given $\alpha$, by the triangle inequality, it is sufficient (up to replacing $s$ by $s+\epsilon$ in \eqref{fractal}) to prove that
 \begin{equation} \label{fractal2} \left(\int \sup_{t \in [0,1]} | E f (x,t) |^2 \, d\mu(x)\right)^{1/2}  \lesssim R^s \|f\|_{L^2(A(R))} c_{\alpha}(\mu)^{1/2}, \end{equation}
for all smooth $f$ supported in $A(R)$ and all Borel measures $\mu$ on $B_n(0,1)$, where $A(R) = B(0,R) \setminus B(0,R/2)$. For a given $f$, if $t(x)$ is a Borel function attaining the sup above, and $\nu$ is the pushforward of $\mu$ under $x \mapsto (x,t(x))$, it is easy to see that $c_{\alpha}(\nu) \leq c_{\alpha}(\mu)$. Therefore, a more general inequality which would imply \eqref{fractal2} is
 \begin{equation} \label{fractal3} \left(\int  | E f (x,t) |^2 \, d\mu(x,t) \right)^{1/2}  \lesssim R^s \|f\|_{L^2(A(R))} c_{\alpha}(\mu)^{1/2}, \end{equation}
 for all smooth $f$ supported in $A(R)$ and all Borel measures $\mu$ on $B_{n+1}(0,1)$. 
 This step is usually referred to as the ``Kolmogorov-Seliverstov-Plessner linearisation argument''. 

By discretisation and pigeonholing, to prove \eqref{fractal3} it is sufficient, up to replacing $s$ with $s+\epsilon$ in \eqref{fractal3}, to show that for any $\gamma \geq 1$, for any smooth $f$ supported in $A(R)$,
\begin{equation} \label{discretised} \left( \int_X |Ef(x,t)|^2 \, dx \, dt \right)^{1/2} \leq \gamma^{1/2} R^{s- \frac{n+1}{2} + \frac{\alpha}{2}} \|f\|_{L^2(A(R))}, \end{equation}
for any disjoint union $X$ of balls of radius $R^{-1}$ in $B_{n+1}(0,1)$ with the property that for any $R^{-1} \leq r \leq 1$ there are at most $\gamma (Rr)^{\alpha}$ balls intersecting any ball of radius $r$. By rescaling and a change of variables, the inequality \eqref{discretised} (over all $X$$, \gamma$, $R$, and $f$) is equivalent to
\begin{equation} \label{rescaleddiscretised} \left( \int_X |Ef(x,t)|^2 \, dx \, dt\right)^{1/2} \leq \gamma^{1/2} R^{s+ \frac{\alpha-n}{2}} \|f\|_{L^2(A(1))}. \end{equation}
for any $\gamma \geq 1$, for any smooth $f$ supported in $A(1)$, and for any disjoint union $X$ of unit balls in $B_{n+1}(0,R)$ with the property that for any $1 \leq r \leq R$, there are at most $\gamma r^{\alpha}$ many unit balls intersecting any $r$-ball. The above discussion yields the following:
\begin{proposition}\label{hamkoleeprop} If $\widetilde{s} \in \mathbb{R}$ and $\alpha \geq 0$ is such that 
\[ \left( \int_X |Ef(x,t)|^2 \, dx \, dt\right)^{1/2} \leq C(\widetilde{s}) \gamma^{1/2} R^{\widetilde{s}} \|f\|_{L^2(A(1))}, \] 
for all smooth $f$ supported in $A(1)$, for any $R,\gamma \geq 1$ and any disjoint union $X$ of unit balls in $B_{n+1}(0,R)$ with the property that for any $1 \leq r \leq R$ there are at most $\gamma r^{\alpha}$ many unit balls intersecting any $r$-ball, then for $s = \widetilde{s} + \frac{n-\alpha}{2}$, 
\[ \mathcal{D}_n(s) \leq \alpha. \]\end{proposition}
The $\epsilon$ losses in the deductions above can be absorbed by $\alpha$, since to show $\mathcal{D}_n(s) \leq \alpha$ it suffices to show that $\mathcal{D}_n(s) \leq \alpha+\epsilon$ for any $\epsilon >0$. By Proposition~\ref{hamkoleeprop}, Theorem~\ref{maintheorem} will be a consequence of the following:
\begin{theorem}  \label{L2estimate} If $n \geq 4$ and $(n+1)/2 < \alpha < n$, then for any $\epsilon >0$, 
\[ \left( \int_X |Ef(x,t)|^2 \, dx \, dt\right)^{1/2} \leq C_{\epsilon} \gamma^{\frac{1}{n-2}} R^{s + \epsilon} \|f\|_{L^2(A(1))}, \]
where
\[ s = \begin{cases}  
\frac{2\alpha-1}{4 n} & \frac{n+1}{2} < \alpha \leq   \frac{n^2-5}{2(n-2)}
\\
\frac{\alpha(n^2-3n+1) - (n^2-2n-1)}{2 n(n-1)(n-3)} & \frac{n^2-5}{2(n-2)} < \alpha < n-1
\\
\frac{ \alpha(2n-1)-(n+1)}{4n(n-1)} & n-1 \leq \alpha \leq  \frac{3n^2-3n+2}{3n-2} 
\\
\frac{ \alpha-2 }{2(n-2)} &  \frac{3n^2-3n+2}{3n-2}  < \alpha < n. \end{cases}   \]
for all smooth $f$ supported in $A(1)$, for any $R \geq 1$ and $\gamma \geq 1$, and any disjoint union $X$ of unit balls in $B_{n+1}(0,R)$ with the property that for any $1 \leq r \leq R$ there are at most $\gamma r^{\alpha}$ many unit balls intersecting any $r$-ball. \end{theorem}

Let $s_n(\alpha)$ be the infimum over all $s$ for which \eqref{fractal3} holds. It is known that $\beta(\alpha, \Gamma^n) = n-2s_n(\alpha)$, where $\beta(\alpha, \Gamma^n)$ is the sup over all $\beta$ for which
\[ \int_{\Gamma^n} \left\lvert \widehat{\mu}(R \xi) \right\rvert^2 \, d\sigma(\xi) \leq C_{\beta} R^{-\beta} \mu(\mathbb{R}^{n+1}) c_{\alpha}(\mu), \]
holds for all Borel $\mu$ on $B_{n+1}(0,1)$ and $R \geq 1$, where $\sigma$ is the $n$-dimensional surface measure on the cone $\Gamma^n$ (equivalently the pushforward of Lebesgue measure on $\mathbb{R}^n$ under $\xi \mapsto (\xi, |\xi|)$). The inequality $\beta(\alpha, \Gamma^n) \geq n-2s_n(\alpha)$ holds by duality, Plancherel, and Cauchy-Schwarz, and the opposite inequality $\beta(\alpha, \Gamma^n) \leq n-2s_n(\alpha)$ follows by reversing this argument; the technique that reverses Cauchy-Schwarz is from \cite[Appendix C]{BBCR} (see also \cite{rogers}). Theorem~\ref{L2estimate} has the following corollary, which strictly supersedes the lower bound for $\beta(\alpha, \Gamma^n)$ from \cite{harris2018} when $\alpha \neq n-1$ and matches it when $\alpha=n-1$.  
\begin{corollary} \label{strongercorollary} If $n\geq 4$ and $(n+1)/2 < \alpha < n$, then 
\[ s_n(\alpha) \leq  \begin{cases}  
\frac{n-\alpha}{2}  + \frac{2\alpha-1}{4 n} & \frac{n+1}{2} < \alpha \leq   \frac{n^2-5}{2(n-2)}
\\
\frac{n-\alpha}{2}  + \frac{\alpha(n^2-3n+1) - (n^2-2n-1)}{2 n(n-1)(n-3)} & \frac{n^2-5}{2(n-2)} < \alpha < n-1
\\
\frac{n-\alpha}{2}  + \frac{ \alpha(2n-1)-(n+1)}{4n(n-1)} & n-1 \leq \alpha \leq  \frac{3n^2-3n+2}{3n-2} 
\\
\frac{n-\alpha}{2}  + \frac{ \alpha-2 }{2(n-2)} &  \frac{3n^2-3n+2}{3n-2}  < \alpha < n. \end{cases}   \]
 Consequently,
\[ \beta(\alpha, \Gamma^n ) \geq 
\begin{cases}  
\alpha- \frac{2\alpha-1}{2n} & \frac{n+1}{2} < \alpha \leq   \frac{n^2-5}{2(n-2)}
\\
\alpha- \frac{\alpha(n^2-3n+1) - (n^2-2n-1)}{ n(n-1)(n-3)} & \frac{n^2-5}{2(n-2)} < \alpha < n-1
\\
\alpha-\frac{ \alpha(2n-1)-(n+1)}{2n(n-1)} & n-1 \leq \alpha \leq  \frac{3n^2-3n+2}{3n-2} 
\\
\alpha- \frac{ \alpha-2 }{n-2} &  \frac{3n^2-3n+2}{3n-2}  < \alpha < n. \end{cases}
\]

When $n=4$ and $19/5 \leq \alpha < 4$, equality holds in the last case, i.e.~$s_4(\alpha) = \frac{6-\alpha}{4}$, and $\beta(\alpha, \Gamma^4) = \frac{\alpha+2}{2}$. \end{corollary} 
\begin{proof} By the explanation leading from \eqref{fractal3} to \eqref{discretised} to \eqref{rescaleddiscretised}, 
\[ s_n(\alpha) \leq s + \frac{n-\alpha}{2}, \]
where the $s$ is from Theorem~\ref{L2estimate}.  This yields the upper bound on $s_n(\alpha)$. The formula $\beta(\alpha, \Gamma^n) = n-2s_n(\alpha)$ then yields the lower bound on $\beta(\alpha, \Gamma^n)$. \end{proof}
\begin{sloppypar}

\begin{proof}[Proof  that Corollary~\ref{strongercorollary} implies Theorem~\ref{maintheorem}]
By Corollary~\ref{strongercorollary} and the discussion leading from \eqref{fractal} to \eqref{fractal2} to \eqref{fractal3}, for $(n+1)/2 < \alpha < n$, for any $\epsilon >0$,
\[\mathcal{D}_n( g(\alpha)+\epsilon ) \leq \alpha, \]
where $g(\alpha)$ is the upper bound for $s_n(\alpha)$ from Corollary~\ref{strongercorollary}, extended continuously to the endpoints. 
Since \mbox{$g: [(n+1)/2, n] \to [1/2,n/4]$} is a strictly decreasing function of $\alpha$, it follows that 
\begin{equation} \label{divbound} \mathcal{D}_n( g(\alpha) ) \leq \alpha \qquad (n+1)/2 < \alpha < n.\end{equation}
Since $g$ is invertible, this yields
\[ \mathcal{D}_n( s ) \leq g^{-1}(s), \qquad 1/2 < s < n/4. \]
It remains to show that $f$ is the inverse to $g$, where $f$ is the  the right-hand side of Theorem~\ref{maintheorem} extended continuously to the endpoints.  Since both $f$ and $g$ are continuous and piecewise affine, to check that $f$ is the inverse to $g$, it suffices (for example) to show that $g(\alpha_i) = s_i$ and that $f(s_i) = \alpha_i$ for all $i \in \{1,2,3,4,5\}$, where 
\[ (\alpha_1, \alpha_2, \alpha_3, \alpha_4, \alpha_5) = \left( \frac{n+1}{2}, \frac{n^2-5}{2(n-2)}, n-1,\frac{3n^2-3n+2}{3n-2}, n \right), \]
and 
\[ \left( s_1, s_2, s_3, s_4, s_5\right) = \left( \frac{n}{4}, \frac{n^3-3n^2+4n-3}{4n(n-2)}, \frac{2n-3}{2(n-1)} ,\frac{4n-5}{2(3n-2)},  \frac{1}{2} \right). \]
This is a straightforward computation, so it is omitted. 
\end{proof}\end{sloppypar}

The sharp values of $s_n(\alpha)$ (and therefore $\beta(\alpha, \Gamma^n)$) are known for all $\alpha$ when $n= 2$ (\cite{erdogan}) and $n=3$ (\cite{chohamlee}), but are open for a large range of $\alpha$ in higher dimensions; see \cite{hamkolee}. Corollary~\ref{strongercorollary} verifies a special case of a conjecture of Cho, Ham, and Lee \cite[p.~62]{chohamlee} when $n=4$, $q=2$, in the very small range $19/5 \leq \alpha < 4$. When $q=2$ and $n \geq 4$, this conjecture can be written as 
\begin{equation} \label{sconj} s_n(\alpha) = \begin{cases} \frac{n-\alpha}{2}  & 0 \leq  \alpha \leq n-2 \\ \frac{n+2-\alpha}{4} & n-2 \leq \alpha < n \\ \frac{n+1-\alpha}{2} & n \leq \alpha \leq n+1, \end{cases}  \end{equation}
or equivalently
\begin{equation} \label{betaconj} \beta(\alpha, \Gamma^n)  = \begin{cases} \alpha & 0 \leq \alpha \leq n-2 \\
\frac{n-2+\alpha}{2} & n-2 \leq \alpha < n \\
\alpha-1 & n \leq \alpha \leq n+1. \end{cases} \end{equation}
Both \eqref{sconj} and \eqref{betaconj} are known when $n \leq \alpha \leq n+1$, or when $\alpha \leq \frac{n-1}{2}$. Cho, Ham, and Lee give examples in \cite{chohamlee} showing that \eqref{sconj} and \eqref{betaconj} would be sharp if true. However, when $n\geq 6$ it is known that \eqref{sconj} and \eqref{betaconj} are false for $\frac{n+1}{2} < \alpha < n-\frac{4}{n-3}$ \cite[Proposition~5.5]{harris}. The following proposition, based on examples from \cite{du}, 
gives a small improvement to the known counterexamples when $\frac{n+1}{2} < \alpha < n - \frac{2}{n-3}$ and $n \geq 6$, and shows that \eqref{sconj} and \eqref{betaconj} are false in this slightly larger range. Outside of this range the best known counterexamples match \eqref{sconj} and \eqref{betaconj}. The proposition below also holds when $n \in \{4,5\}$ but gives no improvement to the upper bound corresponding to \eqref{betaconj}.  

\begin{proposition} \label{counterexample} If $n \geq 4$ and $(n+1)/2 < \alpha < n$ then 
\begin{equation} \label{betalowerbound} \beta(\alpha, \Gamma^n) \leq \begin{cases} \min\left\{\alpha  - \frac{2\alpha -n - 1}{n+1-m}, \alpha-1 + \frac{n+1-\alpha}{n-m} \right\}  & m < \frac{n-1}{2} \\
 \alpha  - \frac{2\alpha -n - 1}{n+1-m} & m \geq \frac{n-1}{2},\end{cases} \end{equation}
 where $m$ is the unique positive integer such that $\alpha \in [n-m, n-m+1)$, or equivalently $m = \lceil n-\alpha \rceil$. 

 Consequently, \[ s_n(\alpha) \geq \frac{n-\alpha}{2} + \begin{cases} \max\left\{\frac{2\alpha -n - 1}{2(n+1-m)}, \frac{1}{2} - \frac{n+1-\alpha}{2(n-m)} \right\}  & m < \frac{n-1}{2} \\
 \frac{2\alpha -n - 1}{2(n+1-m)} & m \geq \frac{n-1}{2},\end{cases} \]
\end{proposition}

 To understand the numerology, the upper bound for $\beta(\alpha, \Gamma^n)$ in Proposition~\ref{counterexample} matches the upper bound for $\beta(\alpha, \mathbb{S}^n)$  from \cite[Theorem~1.1]{du} in the given range of $\alpha$, where $\mathbb{S}^n$ is the unit sphere in $\mathbb{R}^{n+1}$ (the $j$ and $d$ in \cite[Theorem~1.1]{du} correspond to $m+1$ and $n+1$ here, respectively, and the range $\alpha \geq n$ is not relevant here). For the cone this is not particularly strong; one might expect the cone to behave worse than a sphere of equal dimension since locally it is ``less curved'', but the examples proving Proposition~\ref{counterexample} are of a global nature, and it is not even known whether $\beta(\alpha, \Gamma^{n}) \leq \beta(\alpha, \mathbb{S}^n)$ for all $\alpha$ and $n$. 
 
 The examples for $\mathbb{S}^n$ from \cite{du} require a lower bound on the number of ways of writing a large integer $N$ as a sum of $k$ squares, where $k \geq 4$. Since any sequence $N \to \infty$ would suffice, the number theory used in \cite{du} could be replaced by an elementary volume argument. The cone case however requires a simultaneous lower bound for many consecutive $N$ on the number of ways of writing  $N^2$ as a sum of $k$ squares, for $k \geq 3$, and here some number theory seems to be necessary. When $n\geq 7$ the same  bound as in \cite{du} can be used, but when $n=6$, to avoid more advanced number theory (e.g.~class numbers) it is actually important that $N^2$ is not just an integer but the square of an integer, to allow the use of Stieltjes' formula for the number of ways of writing the square of an integer as a sum of three squares.

 The minimum formula in \eqref{betalowerbound} holds when $m \geq \frac{n-1}{2}$ too, but in this range the first term is always the minimum. The case $m \geq \frac{n-1}{2}$ could be written as $m = \frac{n-1}{2}$ when $n$ is odd and $m = \frac{n}{2}$ when $n$ is even, since these are the only possible values of $m$.

It follows from \cite[Theorem~3.1]{harris2026} that lower bounds for $s_n(\alpha)$ automatically imply the same lower bounds for the best possible $s$ in the maximal inequality \eqref{fractal}. It is not yet known whether Stein's maximal principle applies in the fractal setting, but possibly the mass transference technique from \cite{eceizabarrena} could be applied to the counterexample in Proposition~\ref{counterexample}, with some work, to yield  lower bounds for $\mathcal{D}_n(s)$ corresponding to Proposition~\ref{counterexample}, which would disprove the conjecture \eqref{divergencesup} when $n \geq 6$. 

\section{Preliminaries, and multilinear refined Strichartz inequality}

 Let $n \geq 4$. For a function $f$ on $B_n(0,1)$ and $R \geq 1$, let $\{\psi_{\theta}\}_{\theta}$ be a smooth partition of unity subordinate to a boundedly overlapping cover of $B_n(0,1) \setminus B_n(0,1/2)$ by tubes $\theta$ of dimensions $\sim R^{-1/2} \times \dotsb \times R^{-1/2} \times 1$, with long axis in the radial direction. Let $\delta>0$ and, for each $\theta$, let $\{\phi_{\nu}\}_{\nu}$ be a smooth partition of unity subordinate to a boundedly overlapping cover of $\mathbb{R}^n$ by dual $R^{1/2+\delta} \times \dotsb \times R^{1/2+\delta} \times  R^{\delta}$ slabs. Then 
\begin{equation} \label{wavepacket} f = \sum_{(\theta, \nu) \in \mathbb{T}} f_{\theta,\nu}, \end{equation}
where $f_{\theta, \nu} = \psi_{\theta} ( f \ast \mathcal{F}^{-1}(\phi_{\nu}))$ is supported in $\theta$, with Fourier transform essentially supported in $2\nu$ and rapidly decaying outside $\nu$.  
 Here $\mathbb{T}$ is the set of pairs $(\theta, \nu)$. For each $(\theta, \nu)$, the restriction of $Ef_{\theta,\nu}$ to $B_{n+1}(0,R)$ is essentially supported in and rapidly decaying outside a ``plank'' $T_{\theta, \nu}$ of dimensions $\sim R^{1/2+\delta} \times \dotsb \times R^{1/2+\delta} \times R^{\delta} \times R$, where the long direction is normal to the cone $\Gamma^n$ at the lift $\widehat{\theta}$ of $\theta$ into the cone, and the short direction is the radial direction of the cone at $\widehat{\theta}$. To simplify notation, define $T_{\theta, \nu}$ to be the \emph{restriction} of this plank to $B_{n+1}(0,10R)$, so that by definition $T_{\theta, \nu} \subseteq B_{n+1}(0,10R)$. For each fixed $\theta$, these planks form a boundedly overlapping cover of $B_{n+1}(0,R)$ as $\nu$ varies.
 
 \begin{sloppypar} Given $2 \leq m \leq n+1$ and real polynomials $P_1, \dotsc, P_{n+1-m}$ on $\mathbb{R}^{n+1}$, let $Z(P_1, \dotsc, P_{n+1-m})$ be their common zero set, which by convention means $\mathbb{R}^{n+1}$ when $m=n+1$ (i.e.~when the set of polynomials is empty). Call $Z(P_1, \dotsc, P_{n+1-m})$ a transverse complete intersection if either $m=n+1$, or if the vectors $\nabla P_1(x), \dotsc, \nabla P_{n+1-m}(x)$ are linearly independent whenever $x \in Z(P_1, \dotsc, P_{n+1-m})$. Sometimes $Z(P_1, \dotsc, P_{n+1-m})$ will be abbreviated as $Z$.\end{sloppypar}

The proof of Theorem~\ref{refineddecoupling} below is similar to \cite{guthiosevichouwang,GGGHMW}, and the variety case in \cite{duguthli,duguthlizhang}, so only the details where the narrow condition affects the exponent will be emphasised, and the proof will only be sketched so as not to obscure the ideas. In later sections, Theorem~\ref{refineddecoupling} and Theorem~\ref{multilinearstrichartz} will only be used in the range $\alpha \leq n-1$ (which when $n=4$ does not overlap with the range where the result is sharp), so their proofs could be skipped on a first reading. 

In the statement of Theorem~\ref{refineddecoupling} below, given $R \geq 1$ let $\mathbb{T}  = \bigcup_{\theta \in \Lambda} \mathbb{T}_{\theta}$, where $\Lambda$ is a boundedly overlapping cover of the upper light cone $\left\{(\xi, |\xi| ) : 1/2 \leq |\xi| \leq 1 \right\}$ by planks $\theta$ of dimensions $R^{-1/2} \times \dotsb \times R^{-1/2} \times 1 \times R^{-1}$ tangent to the light cone and $\sim 1$ away from the origin. For each $\theta$, $\mathbb{T}_{\theta}$ is boundedly overlapping covering of $\mathbb{R}^{n+1}$ by planks of dimensions $\sim R^{1/2+\delta} \times \dotsb \times R^{1/2+\delta} \times R^{\delta} \times R$ in $B_{n+1}(0,R)$ dual to $\theta$.  

\begin{theorem} \label{refineddecoupling} ($m$-narrow refined decoupling for the cone) Suppose that $2 \leq m \leq n+1$. Let $q_m= \frac{2m}{m-2}$.  For any $\epsilon >0$, there exists $\delta_0>0$ such that the following holds for $0 < \delta < \delta_0$, and for all $A,B, M, R \geq 1$.  Suppose that $f = \sum_{T \in \mathbb{W}} f_T$, where $\mathbb{W} \subseteq \mathbb{T}$ and where each $f_T$ is ``essentially supported'' on a plank $T$ of dimensions $R^{1/2 +\delta} \times \dotsb \times R^{1/2+\delta} \times R^{\delta} \times R^{1+\delta}$ meaning that 
\[ \|f_T\|_{L^{\infty}(\mathbb{R}^{n+1} \setminus T ) } \leq A R^{-10000n}\lVert f \rVert_2, \]
with $\widehat{f_T}$ supported in a dual plank $\theta(T)$ of dimensions $R^{-1/2} \times \dotsb \times R^{-1/2} \times 1 \times R^{-1}$ tangent to the light cone $\left\{(\xi, |\xi| ) : 1/2 \leq |\xi| \leq 1 \right\}$ and $\sim 1$ away from the origin.  Assume that $\|f_T\|_{q_m}$ is constant over $T \in \mathbb{W}$ up to a factor of 2. Suppose that $Y \subseteq B_{n+1}(0,R)$ is a disjoint union of unit balls $Q$, such that for each $Q$ there are at most $M$ planks $T \in \mathbb{W}$ with $2T$ intersecting $Q$. Suppose also that for each $Q$, there is an $m$-dimensional vector space $V=V(Q)$ such that all $T \in \mathbb{W}$ intersecting $Q$ have $\theta(T)$ within a distance $AR^{-1/2}$ of $V$, and suppose that each $V=V(Q)$ is ``$B$-transverse to $\Gamma^n$'', meaning that $|c'|^2-|c''|^2 \geq B^{-1}|c|^2$ for all $c=(c', c'') \in V^{\perp} \subseteq \mathbb{R}^n \times \mathbb{R}$. Then 
\[ \left\lVert f \right\rVert_{L^{q_m}(Y)} \leq \left( \frac{M}{\left\lvert \mathbb{W} \right\rvert } \right)^{\frac{1}{2} - \frac{1}{q_m} }C_{\epsilon, \delta}A^{O(1)} B^{O(1)}   R^{\epsilon}  \left( \sum_{T \in \mathbb{W}} \left\lVert f_T \right\rVert_{q_m}^2 \right)^{1/2}. \]\end{theorem}
\begin{proof}[Sketch of proof] 
It suffices to prove a slightly stronger statement where the unit balls $Q$ are enlarged to balls of radius $K^2$, where $K \ll R^{\delta} \ll R^{\epsilon}$ (say $K=R^{\delta^2})$. This slightly stronger version will be proved via induction on scales. The function $f$ will be decomposed into pieces (to be explained below) which will resemble the claim at scale $\widetilde{R} = R/K^2$ after a Lorentz rescaling which preserves the cone. By pigeonholing, 
\begin{equation}  \label{firststep}  \|f\|_{L^{q_m}(Y)} \lesssim (\log R)^{O(1)} \left\lVert \sum_{\Box \in \mathbb{B}} \sum_{T \in \mathbb{W}_{\Box}} \eta_{Y_{\Box}} f_T \right\rVert_{L^{q_m}(Y')} . \end{equation}
Here the sets $\Box$ have dimensions $\sim R/K \times \dotsb \times R/K \times R \times R/K^2$ in physical space, each corresponding to a $K^{-1} \times \dotsb \times K^{-1} \times K^{-2} \times 1$ plank $\tau=\tau(\Box)$ in frequency space tangent to the forward light cone (the long direction of $\Box$ is the short direction of $\tau$ and vice-versa).  The set $\mathbb{W}_{\Box}$ is the set of $T \in \mathbb{W}$ with $\theta(T) \subseteq \tau(\Box)$ and $T \subseteq \Box$ (some linear algebra shows that the $\mathbb{W}_{\Box}$ are boundedly overlapping, so can be treated as disjoint for the purposes of this sketch). The set $Y_{\Box}$ is a union over sets $Q_{\Box}$ of dimensions $\widetilde{K}^2K \times \dotsb \times \widetilde{K}^2K \times \widetilde{K}^2 \times \widetilde{K}^2K^2$, where $\widetilde{K} = \widetilde{R}^{\delta^2}$ (these will become $\widetilde{K}^2$-balls under rescaling), pigeonholed such that each $Q_{\Box}$ intersects $\sim M'$ many sets $2T$ with $T\in \mathbb{W}_{\Box}$. The function $\eta_{Y_{\Box}}$ is a sum over smooth cutoffs to those $Q_{\Box}\subseteq Y_{\Box}$.   The set $\mathbb{B}$ was pigeonholed such that $\left\lvert \mathbb{W}_{\Box} \right\rvert$ is constant over $\Box \in \mathbb{B}$ up to a factor of 2. Finally, the balls $Q \subseteq Y'$ are chosen such that each $2Q$ intersects $\sim M''$ sets $Y_{\Box}$ as $\Box$ ranges over $\mathbb{B}$.

By the narrow cone decoupling theorem 
(see e.g.~\cite[Theorem~7.1]{gaoliumiaoxi}), and the assumption on $V=V(Q)$ in the theorem statement, decoupling holds with exponent $q_m = \frac{2m}{m-2}$ on each $Q \subseteq Y$. Applying this to \eqref{firststep} gives
\begin{equation} \label{manybrackets} \|f\|_{L^{q_m}(Y)} \lesssim K^{O(\delta)} \left( \sum_Q \left( \sum_{\Box \in \mathbb{B}}\left\lVert  \sum_{T \in \mathbb{W}_{\Box}} \eta_{Y_{\Box}} f_T \right\rVert_{L^{q_m}(2Q)}^2 \right)^{q_m/2} \right)^{1/q_m}. \end{equation}
Since each $Q$ intersects $\sim M''$ many $\Box \in \mathbb{B}$, Hölder's inequality applied to the sum over $\mathbb{B}$ in \eqref{manybrackets} gives
\begin{equation} \label{preinduction} \|f\|_{L^{q_m}(Y)} \lesssim K^{O(\delta)} (M'')^{\frac{1}{2} - \frac{1}{q_m} } \left(   \sum_{\Box \in \mathbb{B}}\left\lVert  \sum_{T \in \mathbb{W}_{\Box}}  f_T \right\rVert_{L^{q_m}(Y_{\Box})}^{q_m}  \right)^{1/q_m}. \end{equation}
After a Lorentz rescaling $L$, for each $\Box$ the right-hand side of \eqref{preinduction} will satisfy the assumptions of the theorem at scale $\widetilde{R} = R/K^2$. Similarly to \cite[p.~628]{duguthli}, it is required to check that the narrow 
condition is preserved under rescaling. This is because, for a given $Q_{\Box}$ (preimage of $\widetilde{K}^2$-ball under rescaling), to make non-negligible contribution it must intersect some fixed $Q \subseteq Y$, so let $V=V(Q_{\Box})=V(Q)$.  Any rescaled plank $\widetilde{T}$ intersecting $L(Q_{\Box})$ must have $2\widetilde{T}$ containing $L(Q_{\Box})$, and so $T:=L^{-1}\left(\widetilde{T}\right)$ has the property that $2T$ intersects $Q$. Therefore, $\theta(T)$ is within $AR^{-1/2}$ of $V$. In particular\footnote{This step is where the ``transverse to $\Gamma^n$'' condition is used. It relies on the elementary property that if $P$ is an $(m-1)$-dimensional affine plane in $\mathbb{R}^d$, $2 \leq m \leq d$, of distance at most $1-B^{-1}$ from the origin, then $\mathcal{N}_{K^{-1}}(P) \cap S^{d-1} \subseteq \mathcal{N}_{C_n B K^{-1}}(P \cap S^{d-1})$ for $K$ large depending on $B$.}, $\theta(T)$ is within $\sim AB R^{-1/2}$ of $V \cap \Gamma^n$ and thus $\theta\left(\widetilde{T}\right) = L^{-1}(\theta(T))$ is within $\sim ABR^{-1/2}K = AB \widetilde{R}^{-1/2}$ of $L^{-1}(V)$. It is also required to check that $L^{-1}(V)$ is ``$\sim B$-transverse to $\Gamma^n$''. By rotational symmetry in the first $n$-coordinates, it may be assumed that the centre of $\tau(\Box)$ is $(e_{n+1}+e_n)/\sqrt{2}$, and in particular
\begin{equation} \label{distcdn} \dist\left(V, (e_n+e_{n+1})/\sqrt{2}\right) < 2K^{-1}, \end{equation} and thus $L$ is the composition of the rotation 
\[ U : e_j \mapsto e_j \quad 1 \leq j \leq n-1, \qquad \frac{e_{n+1}+e_n}{\sqrt{2}} \mapsto e_n, \qquad \frac{e_{n+1}-e_n}{\sqrt{2}} \mapsto e_{n+1}, \]
followed by the parabolic scaling 
\[ (x_1, \dotsc, x_{n-1}, x_n, x_{n+1}) \mapsto (K^{-1}x_1, \dotsc, K^{-1} x_{n-1}, x_n, K^{-2} x_{n+1}) \]
followed by $U^{-1}$. A computation gives
\begin{multline} \label{computation} L(x_1,\dotsc, x_{n-1}, x_n, x_{n+1}) = \bigg(K^{-1}x_1, \dotsc, K^{-1}x_{n-1}, \\
\frac{x_{n+1}+x_n}{2} - K^{-2} \frac{x_{n+1}-x_n}{2}, \frac{x_{n+1}+x_n}{2} + K^{-2} \frac{x_{n+1}-x_n}{2} \bigg). \end{multline}
It follows that if $c=(c',c'') \in (L^{-1}(V))^{\perp} \subseteq \mathbb{R}^n \times \mathbb{R}$ is a unit vector, then $c = L(x)$ for some $x \in V^{\perp}$. Moreover, $|x_{n+1}+x_n| \leq 2K^{-1} |x|$ by \eqref{distcdn}, and thus $|x| \gtrsim K$ by \eqref{computation} and since $L(x)=c$ is a unit vector (note we may assume $A \leq R^{1/100}$). Hence 
\[ |c'|^2- |c''|^2 = K^{-2} \sum_{j=1}^{n-1} |x_j|^2 - K^{-2} (x_{n+1}^2-x_n^2) = K^{-2}(|x'|^2-|x''|^2) \gtrsim B^{-1}. \]
This verifies the transversality condition.

Since the right-hand side of \eqref{preinduction} satisfies the theorem assumptions at scale $\widetilde{R}$, applying the inductive assumption to \eqref{preinduction} gives
\[ \|f\|_{L^{q_m}(Y)} \lesssim R^{\epsilon} K^{O(\delta)-\epsilon} \left(\frac{M'M''}{\left\lvert \mathbb{W}_{\Box} \right\rvert}\right)^{\frac{1}{2} - \frac{1}{q_m} } \left(   \sum_{\Box \in \mathbb{B}}\left( \sum_{T \in \mathbb{W}_{\Box}} \left\lVert f_T \right\rVert_{q_m}^2\right)^{q_m/2}  \right)^{1/q_m}, \]
where the Jacobian factor from the rescaling cancels out since there are $L^{q_m}$ norms on both sides. Similarly to \cite{GGGHMW}, by the dyadically constant property of $\|f_T\|_{q_m}$, the above becomes 
\[ \|f\|_{L^{q_m}(Y)} \lesssim R^{\epsilon}  K^{O(\delta)-\epsilon}\left(\frac{M'M''}{\left\lvert \mathbb{W} \right\rvert}\right)^{\frac{1}{2} - \frac{1}{q_m} } \left( \frac{ \lvert \mathbb{B}\rvert \lvert \mathbb{W}_{\Box} \rvert}{\lvert \mathbb{W} \rvert}\right)^{\frac{1}{q_m}} \left(   \sum_{T \in \mathbb{W}} \left\lVert f_T \right\rVert_{q_m}^2\right)^{\frac{1}{2}}  . \]
The inequalities $\lvert \mathbb{B}\rvert \lvert \mathbb{W}_{\Box} \rvert \lesssim \lvert \mathbb{W} \rvert$ and $M'M'' \lesssim M$ hold by essentially the same argument as in \cite{GGGHMW,guthiosevichouwang}, so the above becomes
 \[ \|f\|_{L^{q_m}(Y)} \lesssim R^{\epsilon} K^{O(\delta)-\epsilon} \left(\frac{M}{\left\lvert \mathbb{W} \right\rvert}\right)^{\frac{1}{2} - \frac{1}{q_m} }  \left(   \sum_{T \in \mathbb{W}} \left\lVert f_T \right\rVert_{q_m}^2\right)^{\frac{1}{2}}. \]
The induction closes since the exponent of $K$ is negative. \end{proof}

The following is a wave version of the multilinear refined Strichartz inequality from \cite[Theorem~4.5]{duguthlizhang}, and a narrow version of the multilinear Strichartz inequality from \cite{harris}. The proof below has many similiarities to the (alternative) proof in \cite{demeterstrichartz} of the multilinear refined Strichartz inequality for the paraboloid. Again, the proof will only be sketched so as not to not obscure the main ideas. 
\begin{theorem} \label{multilinearstrichartz} ($m$-narrow $k$-linear refined Strichartz for the wave equation
) Assume that $2 \leq k \leq m \leq n+1$. Let $q_m = \frac{2m}{m-2}$, let $\epsilon >0$ and $0 < \delta \ll \epsilon$. Let $R, A,K\geq 1$.
Suppose that $Y$ is a disjoint union of $N$ many balls $Q$ of radius $R^{1/2}$, such that for each $j$, $\|Ef_j\|_{L^{q_m}(Q)}$ is constant over $Q \subseteq Y$ up to a factor of 2, with $\{f_1, \dotsc, f_k\}$ $K^{-1}$-transversely supported in $B_n(0,1)\setminus B_n(0,1/2)$. Suppose also that, for each $Q \subseteq Y$, there is an $m$-dimensional vector space $V=V(Q)$ such that for all $j$, $\|Ef_{j,T}\|_{\infty} \leq A R^{-10000n} \|f\|_2$ for all $R^{1/2+\delta} \times \dotsb \times R^{1/2+\delta} \times R^{\delta} \times R$-planks $T$  intersecting $Q$ such that $d(G(\theta(T)), V) > R^{-1/2}$. 
Then
\[ \left\lVert \prod_{j=1}^k |Ef_j|^{1/k} \right\rVert_{L^{q_m}(Y) } \leq C_{\epsilon,A} K^{O(1)} R^{\epsilon} 
R^{ - \frac{n-1}{4} +\frac{n+1}{2q_m}} N^{-\frac{k-1}{k} \left( \frac{1}{2} - \frac{1}{q_m} \right) }  \prod_{j=1}^k \|f_j\|_2^{1/k}. \] \end{theorem}
Above, ``transversely supported'' means that 
\[ \left\lvert G(x_1) \wedge \dotsb \wedge G(x_k) \right\rvert > K^{-1}, \]
whenever $x_j \in \supp \widehat{Ef_j}$, where $G(y)$ gives the unit normal to the light cone at a nonzero point $y$ on the light cone. Note that since $k \geq 2$ and by the transverse support condition, in the proof below it may be assumed that each $V(Q)$ is $\sim K^{-O(1)}$-transverse to $\Gamma^n$ in the sense of Theorem~\ref{refineddecoupling}. 

\begin{proof}[Sketch of proof of Theorem~\ref{multilinearstrichartz}] For each $j$, it may be assumed that the wave packets in $f_j$ have been pigeonholed such that they each have equal $L^2$ norm up to a factor of 2. Let $\mathbb{W}_j$ be the set of pigeonholed wave packets. After pigeonholing the $R^{1/2}$-balls $Q$, it may be assumed that for each $j$, the number of wave packets from $\mathbb{W}_j$ intersecting $Q$ is $\sim M_j$. By Hölder's inequality, followed by $m$-narrow refined decoupling (Theorem~\ref{refineddecoupling}),
\begin{equation} \label{holderstep} \left\lVert \prod_{j=1}^k |Ef_j|^{1/k} \right\rVert_{L^{q_m}(Y) }\lessapprox \prod_{j=1}^k \left( \frac{M_j}{\left\lvert \mathbb{W}_j \right\rvert } \right)^{\frac{1}{k}\left(\frac{1}{2} - \frac{1}{q_m} \right)} \left( \sum_{T \in \mathbb{W}_j } \left\lVert Ef_{j,T} \right\rVert_{q_m}^2  \right)^{\frac{1}{2k} }. \end{equation}
The above used the (blunt) observation that, for a 1-ball inside an $R^{1/2}$-ball, if there are $\sim M_i$ planks intersecting the $R^{1/2}$-ball, then there must be $\lesssim M_i$ planks intersecting the 1-ball. By the Stein-Tomas theorem (interpolating between the cone Stein-Tomas exponent $q=\frac{2(n+1)}{n-1}$ and the trivial case $q=\infty$), or by Bernstein's inequality:
\begin{equation} \label{steintomas} \left\lVert Ef_{j,T} \right\rVert_{q_m} \lessapprox R^{ - \frac{n-1}{4} +\frac{n+1}{2q_m}}\|f_{j,T}\|_2. \end{equation}
By substituting \eqref{steintomas} into \eqref{holderstep} and by orthogonality,
\begin{equation} \label{afterdecoupling} \left\lVert \prod_{j=1}^k |Ef_j|^{1/k} \right\rVert_{L^{q_m}(Y) }\lessapprox R^{ - \frac{n-1}{4} +\frac{n+1}{2q_m}} \prod_{j=1}^k \left( \frac{M_j}{\left\lvert \mathbb{W}_j \right\rvert } \right)^{\frac{1}{k}\left(\frac{1}{2} - \frac{1}{q_m} \right)} \|f_j\|_2^{\frac{1}{k}} . \end{equation}
For each pigeonholed $R^{1/2}$-ball $Q$, 
\[ R^{(n+1)/2} \lesssim \left( \prod_{j=1}^k M_j \right)^{\frac{-1}{k-1}} \int_Q \prod_{j=1}^k \left\lvert \sum_{T_j \in \mathbb{W}_j} \chi_{\mathcal{N}_{R^{1/2}}(T_j)} \right\rvert^{\frac{1}{k-1}}. \]
Summing over $Q$ gives
\[ N R^{(n+1)/2} \lessapprox \left( \prod_{j=1}^k M_j \right)^{\frac{-1}{k-1}} \int_Y \prod_{j=1}^k \left\lvert \sum_{T_j \in \mathbb{W}_j} \chi_{\mathcal{N}_{R^{1/2}}(T_j)}\right\rvert^{\frac{1}{k-1}}. \]
By the Bennett-Carbery-Tao $k$-linear Kakeya inequality in $\mathbb{R}^{n+1}$ (see \cite[Theorem~5.1]{BCT}) applied to the second factor in the right-hand side above,
\[ N \lessapprox 
 \prod_{j=1}^k \left( \frac{M_j}{\left\lvert \mathbb{W}_j \right\rvert} \right)^{\frac{-1}{k-1}}. \] 
 Substituting the above into \eqref{afterdecoupling} finishes the proof. \end{proof}

\section{Weighted broad and linear inequality}

\begin{definition}  Given $n \geq 2$, $A, K \geq 1$ and a function $f$ on $B_n(0,1) \setminus B_{n}(0,1/2)$, form a boundedly overlapping cover of $B_{n}(0,1) \setminus B_n(0,1/2)$ by tubes $\tau$ of dimensions $1 \times K^{-1} \times \dotsb \times  K^{-1}$, where the long direction is radial, and let $f_{\tau} = f \psi_{\tau}$, where $\psi_{\tau}$ is a partition of unity subordinate to the cover. Let $\mathcal{Q}$ be a fixed boundedly overlapping cover of $\mathbb{R}^{n+1}$ by $K^2$-balls $Q$. Given a finite $p \geq 1$, a union $U$ of $K^2$-balls from $\mathcal{Q}$, define the $(n-1)$-broad $L^p$ norm of $Ef$ over $U$ with parameter $A$ by
\[ \|Ef\|_{BL^p_{n-1,A} (U)} = \left( \sum_{Q \subseteq U} \inf_{V_1, \dotsc, V_A \in G(n-2, n+1) } \sup_{\substack{\tau \\ d(G(\tau), V_a) \geq K^{-2} \, \forall a} }  \int_{Q} \left\lvert Ef_{\tau} \right\rvert^p \right)^{1/p}. \] 
The function $U \mapsto \|Ef\|_{BL^p_{n-1,A} (U)}^p$ extends to Borel a measure on $\mathbb{R}^{n+1}$ by defining
\[ \|Ef\|_{BL^p_{n-1,A} (U)}^p = \sum_{Q \in \mathcal{Q}} \frac{ |U \cap Q|}{|Q|} \|Ef\|_{BL^p_{n-1,A} (Q)}^p. \] 
If $p = \infty$, define
\begin{equation} \label{infinityversion} \|Ef\|_{BL^p_{n-1,A} (U)} = \sup_{Q \subseteq U}   \inf_{V_1, \dotsc, V_A \in G(n-2, n+1) } \sup_{\substack{\tau \\ d(G(\tau), V_a) \geq K^{-2} \, \forall a} }  \left\lVert  Ef_{\tau} \right\rVert_{L^{\infty}(Q)}. \end{equation} \end{definition}

Above, $G(\tau)$ is the set of unit normals to the cone at points on the lift of $\tau$ into the cone $\Gamma$. 

The $k$-broad norm could be defined for general $k$, but only the $(n-1)$-broad version will be needed here. Broad norms were introduced in \cite{guth}. 

\begin{theorem} \label{broadweighted} Let $n \geq 4$. Suppose that $n-1 \leq m \leq n+1$, and let $K \geq 1$. For any $\epsilon >0$, there are constants $ 0 < \delta = \delta_{n+1} \ll \delta_n \ll \dotsb \ll \delta_1 \ll \delta_0 \ll \epsilon$, and a large constant $\overline{A}$ so that the following holds for all $1 \leq A \leq \overline{A}$.  Let $\gamma\geq 1$. Let $D_Z$ be any positive integer, and let $Z(P_1, \dotsc, Z_{P_{n+1-m}})$ be a transverse complete intersection, where $\deg(P_j) \leq D_Z$ for all $j \in \{1, \dotsc, n+1-m\}$. Let $f$ be supported in $B_n(0,1) \setminus B_n(0,1/2)$ with $f$ a sum over $f_T$ with $T \in \mathbb{T}_{Z(P_1, \dotsc, Z_{P_{n+1-m}})}$. Here $T \in \mathbb{T}_{Z(P_1, \dotsc, Z_{P_{n+1-m}})}$ means that 
\[ T \subseteq \mathcal{N}_{R^{1/2+\delta_m}}(Z(P_1,\dotsc, P_{n+1-m} ) ), \]
and 
\[  \angle(G(\theta), T_zZ ) < R^{-1/2+\delta_m} \qquad \forall  z \in Z \cap \mathcal{N}_{100R^{1/2+\delta_m}}(T),  \] where $G(\theta)$ is the long direction of $T$. 
Let $\frac{n+1}{2} < \alpha < n$ and let $Y$ be a disjoint union of radius 1 balls in $B_{n+1}(0,R)$, such that for any $1 \leq r \leq R$, there are at most $\gamma r^{\alpha}$ many 1-balls from $Y$ intersecting any ball of radius $r$. Then 
\[ \|Ef\|_{BL^{r_m}_{n-1,A} (Y)} \leq C \gamma^{b_m} R^{m \epsilon} R^{\delta\left( \log \overline{A} - \log A \right)} R^{-a_m}\|f\|_2, \]
where $C= C(\alpha,D_Z,K,\epsilon,m)$,
\[ b_m = \begin{cases} \frac{1}{r_m} - \frac{1}{p_{dec}} & \frac{n+1}{2} < \alpha \leq n-1 \text{ and } m \in \{n,n+1\} \\
\frac{1}{r_m} - \frac{1}{r_{n-1}} & n-1 \leq \alpha < n \text{ or } m=n-1. 
\end{cases} \]
where $\frac{1}{2} - \frac{1}{p_{dec}} = \frac{1}{n-2}$, $r_m$ is defined for $m \in \{n-1,n\}$ by
\[ \frac{1}{2} - \frac{1}{r_m} =   \frac{1}{2m} \qquad  m \in \{n-1,n\}, \]
and 
\begin{equation} \label{rnplusoneformula} \frac{1}{2} - \frac{1}{r_{n+1}}=  \begin{cases} \frac{2\alpha-1}{4\alpha n} & \frac{n+1}{2} < \alpha \leq   \frac{n^2-5}{2(n-2)} \\
\frac{\alpha(n^2-3n+1) - n^2+2n+1}{2\alpha n(n-1)(n-3)} &   \frac{n^2-5}{2(n-2)} < \alpha < n-1 \\
\frac{\alpha(2n-1)-(n+1)}{4\alpha n (n-1) } & n-1 \leq \alpha < n, \\
\end{cases} \end{equation}
and $a_m$ is defined for $m \in \{n-1,n+1\}$ by
\[ a_m =   \frac{1}{2}\left( \frac{1}{2} - \frac{1}{r_m} \right)\left( n+1-m \right) \qquad m \in \{n-1,n+1\},\] 
and 
\begin{equation} \label{andefn} a_n =  \begin{cases} \frac{1}{4n} & \frac{n+1}{2} < \alpha \leq   \frac{n^2-5}{2(n-2)} \\
\frac{n^2 -  \alpha(n-2) - 2n-1}{2n(n-1)(n-3)} &  \frac{n^2-5}{2(n-2)} < \alpha \leq n-1 \\
\frac{n+1-\alpha}{4n(n-1)} & n-1 \leq \alpha < n. \end{cases} \end{equation}
\end{theorem}  
\begin{proof} By Hölder's inequality for the broad norms (see \cite{guth}), it will suffice to prove the theorem with $r_m$ replaced by $r=r_m+\delta$. It will be shown that the inequality holds via induction on $R \geq 1$. The case $R \leq 100$ is trivial, so let $R \geq 100$ and assume inductively that the theorem holds for all radii smaller than $R/2$. Induction is also done on $A$ and $m$; for $A \lesssim 1$ the result is trivial since $\overline{A}$ can be chosen very large depending on $\delta$. Therefore, given $R \geq 1$ and $1 \leq A \leq \overline{A}$, it may be assumed that the result holds for all $\widetilde{R} \leq R/2$ and $1 \leq \widetilde{A} \leq A/2$, and $\widetilde{m} < m$. Under these assumptions it will be shown for the listed parameters, and as explained above it may be assumed that $A \geq 4$. It may be assumed that $K \leq R^{\delta}$.  

Let $D$ be a large integer to be chosen later, depending on $\epsilon$ and with $D \geq (n+1)D_Z$. As in \cite{guth}, the algebraic case is when there is a transverse complete intersection $Z_{m-1} =Z(P_1, \dotsc, P_{n+1-m}, P) \subseteq Z(P_1, \dotsc, P_{n+1-m})$ of dimension $m-1$ and degree at most $D$, such that 
\[ \|Ef\|_{BL^{r}_{n-1,A} (Y)} \lesssim \|Ef\|_{BL^{r}_{n-1,A}\left(Y \cap \mathcal{N}_{R^{1/2+ \delta_m}}\left(Z_{m-1}\right)\right)}. \]
The other case is the cellular case (also called non-algebraic case).

In the cellular case, by the same argument as in \cite[Section~8.1]{guth}, there is a polynomial $P$ on $\mathbb{R}^{n+1}$, 
where $P$ has degree at most $D$, such that $\mathbb{R}^n \setminus Z(P) = \bigcup_{i \in \mathcal{I}} O_i$ where $\left\lvert \mathcal{I} \right\rvert \sim D^m$ and the $O_i$ are disjoint connected open sets, such that for at least $90\%$ of the $i$'s,
\begin{equation} \label{cellcase} \|Ef\|_{BL^{r}_{{n-1},A} (Y)}^{r} \sim D^{m} \|Ef\|_{BL^{r}_{{n-1},A}\left(Y \cap \widetilde{O}_i\right)}^{r}, \end{equation}
where 
\[ \widetilde{O}_i = O_i \setminus \mathcal{N}_{R^{1/2+\delta}}(Z(P)). \]
Using the wave packet decomposition from \eqref{wavepacket}, let 
\[ f = \sum_{(\theta, \nu) \in \mathbb{T}} f_{\theta,\nu}, \qquad f_i = \sum_{(\theta, \nu) \in \mathbb{T}_i} f_{\theta,\nu}, \]
where $\mathbb{T}_i$ consists of those pairs $(\theta, \nu) \in \mathbb{T}$ for which $T_{\theta, \nu}$ intersects $\widetilde{O_i}$. By essential orthogonality of wave packets (similarly to~\cite[Eq.~3.1]{guth}),
\begin{multline} \label{manyi} \sum_i \|f_i\|_2^2 \lesssim \sum_i \sum_{(\theta, \nu) \in \mathbb{T}_i} \|f_{\theta, \nu}\|_2^2 = \sum_{(\theta, \nu) \in \mathbb{T}} \sum_{\substack{i \\ T_{\theta, \nu} \cap \widetilde{O}_i} \neq \emptyset } \|f_{\theta, \nu}\|_2^2 \\
\lesssim D  \sum_{(\theta, \nu) \in \mathbb{T}} \|f_{\theta, \nu}\|_2^2  \lesssim D \|f\|_2^2. \end{multline}
The second line used that, for a given plank $T_{\theta,\nu}$, if $T_{\theta, \nu}$ intersects $\widetilde{O}_i$, then since the tubes $T_{\theta,\nu}$ have length $R^{1/2+\delta}$ in the medium direction, the centre line of $T_{\theta, \nu}$ intersects $O_i$. This line can intersect at most $D+1$ different cells $O_i$, since the restriction of $P$ to this line is a polynomial of degree at most $D$, which may be assumed not identically zero since otherwise it would intersect \emph{no} cells $O_i$, and this line intersecting more than $D+1$ cells would imply that this non-constant polynomial of degree at most $D$ has more than $D$ zeroes.

The inequality \eqref{manyi} implies that 
\begin{equation} \label{l2bound} \|f_i\|_2^2 \lesssim D^{1-m} \|f\|_2^2, \end{equation}
for at least $90\%$ of the $i'$s. By \eqref{cellcase}, for $90\%$ of the $i$'s.
\begin{equation} \label{usedveryshortly} \|Ef\|_{BL^{r}_{{n-1},A} (Y)}^{r} \lesssim D^{m} \|Ef_i\|_{BL^{r}_{{n-1},A} (Y)}^{r}, \end{equation}
where the very small error term can be assumed to not dominate, since the conclusion of the theorem holds trivially whenever it dominates. By translation, the conclusion of the theorem in $B_n(0,R)$ automatically implies the same in any ball of radius $R$. Using \eqref{l2bound} and \eqref{usedveryshortly}, covering $Y$ with $\sim 1$ many balls of radius $R/2$ and applying the theorem at the assumed scale $R/2$, there is at least one $i$ such that
\begin{multline*}  \|Ef\|_{BL^{r}_{{n-1},A} (Y)}^{r} \lesssim D^{m} \left(\gamma^{b_m} R^{m \epsilon} R^{\delta\left( \log \overline{A} - \log A \right)} R^{-a_m}\right)^r\|f_i\|_2^{r} \\
\lesssim D^{m-r\left( \frac{m-1}{2} \right)} \left(\gamma^{b_m} R^{m \epsilon} R^{\delta\left( \log \overline{A} - \log A \right)} R^{-a_m}\right)^r\|f\|_2^{r}. \end{multline*}
The definition of $r$ and $r_m$, and the restriction on the range of $\alpha$, imply that $r > \frac{2m}{m-1}$ for all $m \in \{n-1,n,n+1\}$ (when $m=n+1$, it suffices to check that $\frac{1}{2} - \frac{1}{r_{n+1}} > \frac{1}{2(n+1)}$  for $\alpha >\frac{n+1}{2}$, which holds since the right-hand side of \eqref{rnplusoneformula} is increasing in $\alpha$, and equality holds at $\alpha=\frac{n+1}{2}$). Taking $D$ sufficiently large\footnote{the choice of $D$ depends on $\delta$, but $\delta$ (chosen implicitly) depends only on $\epsilon$, which means that $D$ depends only on $\epsilon$ and $D_Z$.} to eliminate the implicit constants finishes the proof in the cellular (non-algebraic) case.

In the algebraic case, there is a polynomial $P$ on $\mathbb{R}^{n+1}$, of degree at most $D$, such that $Z(P_1,\dotsc, P_{n+1-m}, P)$ is a transverse complete intersection of dimension $m-1$, and such that 
\[ \|Ef\|_{BL^{r}_{{n-1},A} (Y)} \lesssim \|Ef\|_{BL^{r}_{{n-1},A}\left(Y \cap \mathcal{N}_{R^{1/2+ \delta_m}}\left(Z\left(P_1, \dotsc, P_{n+1-m}, P\right)\right)\right)}. \]
For simplicity abbreviate $Z = Z\left(P_1, \dotsc, P_{n+1-m}, P\right)$. Define $\rho$ by 
\[ \rho^{\frac{1}{2} + \delta_{m-1}} = R^{\frac{1}{2} + \delta_m}, \]
so that $\rho \ll R$, and cover the ball $B_{n+1}(0,R)$ by balls $\{B_j\}_j$ of radius $\rho$. Then
\begin{equation} \label{jsum} \|Ef\|_{BL^{r}_{{n-1},A}\left(Y \cap \mathcal{N}_{R^{1/2+ \delta_m}}(Z)\right)}^r \leq \sum_j \|Ef\|_{BL^{r}_{{n-1},A}\left(Y \cap B_j \cap \mathcal{N}_{R^{1/2+ \delta_m}}(Z)\right)}^r. \end{equation}
For each $j$, using the triangle inequality for broad norms (\cite[Lemma~4.1]{guth}), write
\begin{multline} \label{triangleapp} \|Ef\|_{BL^{r}_{{n-1},A}\left(Y \cap B_j \cap \mathcal{N}_{R^{1/2+ \delta_m}}(Z)\right)} 
\lesssim  \|Ef_{j,tan}\|_{BL^{r}_{{n-1},A/2}\left(Y \cap B_j \cap \mathcal{N}_{R^{1/2+ \delta_m}}(Z)\right)} \\
+ \|Ef_{j,tr}\|_{BL^{r}_{{n-1},A/2}\left(Y \cap B_j \cap \mathcal{N}_{R^{1/2+ \delta_m}}(Z)\right)}, \end{multline}
where 
\[ f_{j,tan} = \sum_{T \in \mathbb{T}_{j,tan}} f_T, \qquad  f_{j,tr} = \sum_{T \in \mathbb{T}_{j,tr}} f_T,  \]
 and
\[ \mathbb{T}_j = \{ T \in \mathbb{T} : T \cap B_j \cap \mathcal{N}_{R^{1/2+ \delta_m}}(Z)\neq \emptyset \}, \]
where
\begin{multline} \label{angleconditionweak} \mathbb{T}_{j,tan} =  \{ T \in \mathbb{T}_j :T \cap B_j \subseteq \mathcal{N}_{R^{1/2+\delta_{m}}(Z)}, \\
 \angle(T, T_zZ ) < \rho^{-1/2+\delta_{m-1}}  \forall  z \in Z \cap B_j \cap \mathcal{N}_{100R^{1/2+\delta_m}}(T) \},\end{multline}
and 
\[ \mathbb{T}_{j,tr} = \mathbb{T}_j \setminus \mathbb{T}_{j,tan}. \]

The negligible error term missing from \eqref{triangleapp} was omitted as it may be assumed to not dominate. By \eqref{jsum} and \eqref{triangleapp}, either 
\begin{equation} \label{jsumtan} \|Ef\|_{BL^{r}_{{n-1},A}\left(Y \cap \mathcal{N}_{R^{1/2+ \delta_m}}(Z)\right)}^r \leq \sum_j \|Ef_{j,tan}\|_{BL^{r}_{{n-1},A/2}\left(Y \cap B_j \cap \mathcal{N}_{R^{1/2+ \delta_m}}(Z)\right)}^r, \end{equation}
or 
\begin{equation} \label{jsumtr} \|Ef\|_{BL^{r}_{{n-1},A}\left(Y \cap \mathcal{N}_{R^{1/2+ \delta_m}}(Z)\right)}^r \leq \sum_j \|Ef_{j,tr}\|_{BL^{r}_{{n-1},A/2}\left(Y \cap B_j \cap \mathcal{N}_{R^{1/2+ \delta_m}}(Z)\right)}^r. \end{equation}
The tangential sub-case is where \eqref{jsumtan} holds, and the transverse sub-case is where \eqref{jsumtr} holds. 

Consider the tangential sub-case,  
and assume first that $m=n-1$. For any $j$ and any $1$-ball $Q \subseteq Y$ with $Q \cap B_j \cap \mathcal{N}_{R^{1/2+ \delta_m}}(Z) \neq \emptyset$,  there is $z \in Z \cap 2B_j$ with $\dist(z,Q)  < R^{1/2+\delta_m}$. Therefore, for any $T \in \mathbb{T}_{j,tang}$ with $T \cap Q \neq \emptyset$,  $z \in \mathcal{N}_{100R^{1/2+\delta_m}}(T) \cap 2B_j \cap Z$, and therefore $\angle (T, T_zZ) \leq \rho^{-1/2+\delta_{m-1}}$ by \eqref{angleconditionweak}. If $V$ is defined to be the $(m-1)$-plane $T_zZ$, and if $R$ is assumed so large that $\rho^{-1/2+\delta_{m-1}} < K^{-2}$, this implies that
\[ \sup_{\tau : d(G(\tau), V) \geq K^{-2} } \int_Q \left\lvert E(f_{j,tan})_{\tau} \right\rvert^{r} \lesssim R^{-10000}. \]
Since $m={n-1}$, it follows that 
\[ \inf_{V \in G(n-2,n+1) } \sup_{\tau : \dist(G(\tau), V) \geq K^{-2} } \int_Q \left\lvert E(f_{j,tan})_{\tau} \right\rvert^{r} \leq CR^{-10000}. \]
Therefore, by definition of the broad norm (note $A \geq 2$), this shows that 
\[ \|Ef_{j,tan}\|_{BL^{r}_{{n-1},A/2}\left(Y \cap B_j \cap \mathcal{N}_{R^{1/2+ \delta_m}}(Z)\right)}^r \]is negligible when $m={n-1}$, and this finishes the proof of the tangential sub-case when $m={n-1}$. 

Now consider the case that either $m=n$ or $m=n+1$ (still in the tangential sub-case). 

If $n-1 \leq \alpha < n$, 
then by Hölder's inequality for the broad norm (see \cite[Lemma~4.2]{guth} with\footnote{Strictly speaking, this doesn't follow from \cite[Lemma~4.2]{guth} directly, since the two functions that Hölder's inequality is applied to in \cite{guth} are both powers of $|Ef|$. But if one modifies the proof in \cite{guth} so that the first application of Hölder's inequality is applied to $|Ef_{\tau}|$ and $\chi_{U \cap B_{K^2}}$, the version of Hölder's inequality claimed above still holds.} $A_1 = 0$, $A_2=A/2$), 
\begin{multline} \label{holderbroad} \|Ef_{j,tan}\|_{BL^{r_m}_{{n-1},A/2}\left(Y \cap B_j \cap \mathcal{N}_{R^{1/2+ \delta_m}}(Z)\right)} \lesssim \\
\left\lvert Y \cap B_j \cap \mathcal{N}_{R^{1/2+ \delta_m}}(Z) \right\rvert^{\frac{1}{r_{m}} - \frac{1}{r_{m-1}} } \|Ef_{j,tan}\|_{BL^{r_{m-1}}_{{n-1},A/2}\left(Y \cap B_j\right)}. \end{multline}
The above used that $r_{n-1} > r_n > r_{n+1}$. To see that $r_n > r_{n+1}$, it is enough to check $\frac{1}{2}-\frac{1}{r_{n+1}} < \frac{1}{2} - \frac{1}{r_n}$,  which holds since substituting $\alpha=n+1$ into the third formula in the right-hand side of \eqref{rnplusoneformula} gives $\frac{1}{2n} = \frac{1}{2} - \frac{1}{r_n}$, and the  right-hand side of \eqref{rnplusoneformula} is increasing in $\alpha$.

\begin{sloppypar}By Wongkew's theorem \cite{wongkew}, $\mathcal{N}_{R^{1/2+\delta_m}}(Z) \cap B_j$ can be covered by $D^{O(1)} R^{(m-1)(1/2-\delta_m)}$ many balls of radius $R^{1/2+\delta_m}$, and each of these balls intersects $\lesssim \gamma R^{(1/2+\delta_m)\alpha}$ many radius 1 balls from $Y$. Since $Y$ also has $\lesssim \gamma R^{\alpha}$ many radius 1 balls in $B_{n+1}(0,R)$, this yields 
\begin{equation} \label{wongkewfractal} \left\lvert Y \cap B_j \cap \mathcal{N}_{R^{1/2+ \delta_m}}(Z) \right\rvert \lesssim D^{O(1)} R^{2\delta_m } \gamma  \min\left\{  R^{\frac{m-1+\alpha}{2} }, R^{\alpha} \right\}. \end{equation}
If $m=n$, then the first term is the minimum by the assumption $\alpha >n-1$. The second term is the minimum when $m=n+1$, since  $\alpha <n$. This bounds the first factor in \eqref{holderbroad}.

To bound the second factor in \eqref{holderbroad},  $f_{j,tan}$ satisfies the hypothesis at dimension $m-1$ and scale $\rho$ by same reasoning as in \cite{ouwang}. By the inductive assumption on $m$, 
\begin{multline} \label{inductivetang} \|Ef_{j,tan}\|_{BL^{r_{m-1}}_{{n-1},A/2}\left(Y \cap B_j\right)} \\
\lesssim \gamma^{\frac{1}{r_{m-1}}-\frac{1}{r_{n-1}}} R^{(m-0.9) \epsilon} R^{\delta\left( \log \overline{A} - \log (A/2) \right)} R^{-a_{m-1}}\|f_{j,tan}\|_2. \end{multline}
Applying \eqref{wongkewfractal} and \eqref{inductivetang} to each factor in \eqref{holderbroad} gives
\begin{multline} \label{tangbound} \|Ef_{j,tan}\|_{BL^{r_m}_{{n-1},A/2}\left(Y \cap B_j \cap \mathcal{N}_{R^{1/2+ \delta_m}}(Z)\right)} \\
\lesssim \left( D^{O(1)} R^{2\delta_m} \gamma  \min\left\{  R^{\frac{m-1+\alpha}{2} }, R^{\alpha} \right\} \right)^{\frac{1}{r_{m}} - \frac{1}{r_{m-1}} } \\
\times \gamma^{\frac{1}{r_{m-1}}-\frac{1}{r_{{n-1}}}} R^{(m-0.9) \epsilon} R^{\delta\left( \log \overline{A} - \log (A/2) \right)} R^{-a_{m-1}}\|f_{j,tan}\|_2 \\
\leq D^{O(1) } R^{(m-0.8)\epsilon} \gamma^{\frac{1}{r_m}-\frac{1}{r_{n-1}} } \\
\times R^{\delta(\log \overline{A} - \log A )} R^{-a_{m-1} + \left( \frac{1}{r_{m}} - \frac{1}{r_{m-1}} \right)\min\left\{\frac{m-1+\alpha}{2} , \alpha \right\} } \|f_{j, tan}\|_2.  \end{multline}
It will be shown that the exponent of $R$ satisfies 
\begin{equation} \label{tobeshown} -a_{m-1} + \left( \frac{1}{r_{m}} - \frac{1}{r_{m-1}} \right)\min\left\{\frac{m-1+\alpha}{2} , \alpha \right\}  = -a_{m}. \end{equation}
When $m=n+1$ this will be due to the definition of $r_{n+1}$, and when $m=n$ it will be due to the definition of $a_n$. If $m=n$, then since $a_{n-1} =  \frac{1}{2}\left( \frac{1}{2} - \frac{1}{r_{n-1}} \right)\left( (n+1)-(n-1) \right)$,  $\frac{1}{2}-\frac{1}{r_{n-1}} = \frac{1}{2(n-1)}$, and $\frac{1}{2} - \frac{1}{r_n} = \frac{1}{2n}$, using $n-1 \leq \alpha <n$ gives that \eqref{tobeshown} is equivalent to $a_n = \frac{n+1-\alpha}{4n(n-1)}$ when $\alpha \geq n-1$, and this follows from the definition of $a_n$ in \eqref{andefn}.  If $m=n+1$, then since $a_n = \frac{n+1-\alpha}{4n(n-1)}$, $a_{n+1}=0$, and $\frac{1}{2}-\frac{1}{r_n}=\frac{1}{2n}$, by the assumption $n-1 \leq \alpha < n$,  \eqref{tobeshown} is equivalent to $\frac{1}{2} - \frac{1}{r_{n+1}} = \frac{\alpha(2n-1)-(n+1)}{4\alpha n (n-1) }$, which holds by the definition of $r_{n+1}$ in \eqref{rnplusoneformula}. Therefore, \eqref{tangbound} becomes
\begin{multline*} \|Ef_{j,tan}\|_{BL^{r_m}_{{n-1},A/2}\left(Y \cap B_j \cap \mathcal{N}_{R^{1/2+ \delta_m}}(Z)\right)} \\
\\
\lesssim  D^{O(1) } R^{(m-0.7)\epsilon} \gamma^{\frac{1}{r_m}-\frac{1}{r_{n-1}} }  R^{\delta(\log \overline{A} - \log A )} R^{-a_m} \|f_{j, tan}\|_2. \end{multline*}
Substituting into \eqref{jsumtan}, using $R/\rho \ll R^{0.1 \epsilon}$, gives 
\[ \|Ef\|_{BL^{r}_{{n-1},A} (Y \cap \mathcal{N}_{R^{1/2+ \delta_m}}(Z))} \lesssim  D^{O(1) } R^{(m-0.6)\epsilon} \gamma^{\frac{1}{r_m}-\frac{1}{r_{n-1}} }  R^{\delta(\log \overline{A} - \log A )} R^{-a_m} \|f\|_2. \]
This finishes the proof of the tangential sub-case when $n-1 \leq \alpha < n$. \end{sloppypar} 

If $(n+1)/2 < \alpha \leq n-1$ (still in the tangential sub-case) and $m=n$, let $q = \frac{2(n-1)}{n-3}$, and define $p_{n-1}$ by 
\begin{equation} \label{pexponent} \frac{1}{2}-\frac{1}{p_{n-1}} = \frac{1}{n-1} \left( \frac{1}{2} - \frac{1}{q}\right) =  \frac{1}{(n-1)^2};\end{equation}
the reason for choosing these exponents will be clear in a moment. Suppose that $m=n$. Write $g = f_{j,tan}$. By the definition of the broad norm, and some pigeonholing, there exist $K^{-2}$-transversely supported $\{g_1, \dotsc, g_{n-1}\}$, where each $g_k$ is a restriction of $g$ to some $\tau_k$, such that 
\begin{multline} \label{broadmultilinear} \|Eg\|_{BL^{p_{n-1}}_{{n-1},A/4}\left(Y \cap B_j \cap \mathcal{N}_{R^{1/2+ \delta_m}}(Z)\right)} \\\lesssim K^{O(1)} \left\lVert \prod_{k=1}^{n-1} \left\lvert Eg_k \right\rvert^{\frac{1}{n-1}} \right\rVert_{L^{p_{n-1}}\left(Y \cap B_j \cap \mathcal{N}_{R^{1/2+ \delta_m}}(Z)\right)}. \end{multline}

After dyadically pigeonholing the right-hand side of \eqref{broadmultilinear} into $R^{1/2}$-balls according to $\|Eg_1\|_{L^q(\widetilde{Q})}$, then according to $\|Eg_2\|_{L^q(\widetilde{Q})}$, and so on without relabelling, and by losing a factor at most $\log(R)^{(n-1)/p_{n-1}}$ in \eqref{broadmultilinear}, it may be assumed that the unit balls in $Y$ are arranged in $N$ many $R^{1/2}$-balls $\widetilde{Q}$, such that for any $k$, $\|Eg_k\|_{L^q(\widetilde{Q})}$ is constant over $\widetilde{Q}$ up to a factor of 2. By \eqref{broadmultilinear}, Hölder's inequality, and then the $(n-1)$-narrow version of Theorem~\ref{multilinearstrichartz} with $k=n-1$ and $V=V(Q)$ the tangent space to the $(n-1)$-variety $Z$ at $Q$,
\begin{multline*} \|Eg\|_{BL^{p_{n-1}}_{{n-1},A/4}\left(Y \cap B_j \cap \mathcal{N}_{R^{1/2+ \delta_m}}(Z)\right)} \\
\lesssim K^{O(1)} R^{\epsilon/2}\left( N \gamma R^{\alpha/2} \right)^{\frac{1}{p_{n-1}}-\frac{1}{q}} \left\lVert \prod_{k=1}^{n-1} |Eg_k|^{\frac{1}{n-1}} \right\rVert_{L^q(Y \cap B_j)} \\
\lesssim  R^{\epsilon} \left( N \gamma R^{\alpha/2} \right)^{\frac{1}{p_{n-1}}-\frac{1}{q}} R^{-\left(\frac{1}{2}-\frac{1}{q} \right)} N^{-\left( \frac{1}{2} - \frac{1}{q} \right) \left( \frac{n-2}{n-1} \right)} \prod_{k=1}^{n-1} \|g_k\|_2^{\frac{1}{n-1}}. \end{multline*} 
The exponent $p_{n-1}$ was chosen in \eqref{pexponent} make the power of $N$ cancel, so this becomes 
\begin{multline} \label{factoronebound} \|Ef_{j,tan}\|_{BL^{p_{n-1}}_{{n-1},A/4}\left(Y \cap B_j \cap \mathcal{N}_{R^{1/2+ \delta_m}}(Z)\right)} \\
\lesssim R^{\epsilon} \gamma^{\frac{1}{p_{n-1}} - \frac{1}{q}} R^{\frac{\alpha}{2}\left(\frac{1}{p_{n-1}}-\frac{1}{q}\right)} R^{-\left(\frac{1}{2}-\frac{1}{q} \right)}  \|f_{j,tan} \|_2. \end{multline}
Define $\theta \in (0,1)$ by 
\[ \theta = \frac{n^2-4n+1}{n(n-3)}, \qquad 1-\theta =\frac{n-1}{n(n-3)}, \]
so that 
\[ \frac{1}{2} - \frac{1}{r_n} = (1-\theta) \left( \frac{1}{2} - \frac{1}{p_{n-1}} \right) + \theta\left( \frac{1}{2} - \frac{1}{r_{n-1}}\right), \]
and
\[ \frac{1}{r_n} = \frac{1-\theta}{p_{n-1}}  + \frac{\theta}{r_{n-1}}. \]
Then by Hölder's inequality for the broad norm \cite[Lemma~4.2]{guth}, followed by applying \eqref{factoronebound} to the first factor and the $(n-1)$-variety case to the second factor,
\begin{multline*} \|Ef_{j,tan}\|_{BL^{r_{n}}_{{n-1},A/2}\left(Y \cap B_j \cap \mathcal{N}_{R^{1/2+ \delta_m}}(Z)\right)} \\
\lesssim  \|Ef_{j,tan}\|_{BL^{p_{n-1}}_{{n-1},A/4}\left(Y \cap B_j \cap \mathcal{N}_{R^{1/2+ \delta_m}}(Z)\right)}^{1-\theta} \|Ef_{j,tan}\|_{BL^{r_{n-1}}_{{n-1},A/4}\left(Y \cap B_j \cap \mathcal{N}_{R^{1/2+ \delta_m}}(Z)\right)}^{\theta} \\
 \lesssim \gamma^{\left(\frac{1}{p_{n-1}} - \frac{1}{q}\right)(1-\theta)} R^{(n-1)\epsilon} \times \\
 R^{(1-\theta)\left[ \frac{\alpha}{2} \left( \left(\frac{1}{2}-\frac{1}{q}\right) - \left(\frac{1}{2} -\frac{1}{p_{n-1}}\right) \right) - \left( \frac{1}{2} - \frac{1}{q} \right) \right]-\theta a_{n-1} } \|f\|_2.\end{multline*}
 After some simplification of the exponent of $R$, using that $a_{n-1} = \frac{1}{2(n-1)}$, this becomes
\begin{multline} \label{nearlydone} \|Ef_{j,tan}\|_{BL^{r_{n}}_{{n-1},A/2}\left(Y \cap B_j \cap \mathcal{N}_{R^{1/2+ \delta_m}}(Z)\right)} \\
 \lesssim  \gamma^{\left(\frac{1}{p_{n-1}} - \frac{1}{q}\right)(1-\theta)} R^{(n-1)\epsilon+\delta \log \left(\overline{A}/A\right)}  R^{\frac{ \alpha(n-2) + 2n+1-n^2}{2n(n-1)(n-3)}} \|f\|_2.\end{multline}
 The power of $\gamma$ in \eqref{nearlydone} satisfies
 \begin{multline} \label{gammapower} \left(\frac{1}{p_{n-1}} - \frac{1}{q}\right)(1-\theta) = \frac{n-1}{n(n-3)} \left( \frac{1}{n-1} - \frac{1}{(n-1)^2} \right) \\
 \leq \frac{1}{n-2} - \frac{1}{2n} =  \frac{1}{r_n} - \frac{1}{p_{dec}}, \end{multline}
the inequality step in \eqref{gammapower} being equivalent to $n(n-1)(n-3) \geq 2$, which holds since $n \geq 4$. Therefore, \eqref{nearlydone} becomes
\begin{multline} \label{tbs} \|Ef_{j,tan}\|_{BL^{r_{n}}_{{n-1},A/2}\left(Y \cap B_j \cap \mathcal{N}_{R^{1/2+ \delta_m}}(Z)\right)} \\
 \lesssim  \gamma^{\frac{1}{r_n} - \frac{1}{p_{dec}}} R^{(n-1)\epsilon+\delta \log \left(\overline{A}/A\right)}  R^{\frac{ \alpha(n-2) + 2n+1-n^2}{2n(n-1)(n-3)}} \|f\|_2. \end{multline}
When $(n+1)/2 < \alpha \leq n-1$, the definition of $a_n$ in \eqref{andefn} is equivalent to 
 \begin{multline*} a_n = \min\left\{ \frac{n^2 -  \alpha(n-2) - 2n-1}{2n(n-1)(n-3)}, \frac{1}{4n} \right\} \\
= \begin{cases} 
\frac{1}{4n} & \frac{n+1}{2} < \alpha \leq   \frac{n^2-5}{2(n-2)} \\
\frac{n^2 -  \alpha(n-2) - 2n-1}{2n(n-1)(n-3)} &  \frac{n^2-5}{2(n-2)} < \alpha \leq n-1,  \\\end{cases} \end{multline*} 
so \eqref{tbs} implies that
\[\|Ef_{j,tan}\|_{BL^{r_{n}}_{{n-1},A/2}\left(Y \cap B_j \cap \mathcal{N}_{R^{1/2+ \delta_m}}(Z)\right)} \\
 \lesssim  \gamma^{\frac{1}{r_n} - \frac{1}{p_{dec}}} R^{(n-1)\epsilon+\delta \log \left(\overline{A}/A\right)}  R^{-a_m} \|f\|_2. \]
 Substituting into \eqref{jsumtan} verifies the tangential sub-case when $m=n$ and $(n+1)/2 < \alpha \leq n-1$.

If $m=n+1$ and $(n+1)/2 < \alpha \leq n-1$, by Hölder's inequality for the broad norm (as in \eqref{tangbound}) and by taking the $\alpha$ term in the minimum in \eqref{tangbound},  
\begin{multline} \label{tbs2} \|Ef_{j,tan}\|_{BL^{r_m}_{{n-1},A/2}\left(Y \cap B_j \cap \mathcal{N}_{R^{1/2+ \delta_m}}(Z)\right)} \\
\lesssim  D^{O(1) } R^{(m-0.9)\epsilon} \gamma^{\frac{1}{r_m}-\frac{1}{p_{dec}} }
 R^{\delta(\log \overline{A} - \log A )} R^{-a_{m-1} + \left( \frac{1}{r_{m}} - \frac{1}{r_{m-1}} \right)\alpha} \|f_{j, tan}\|_2.  \end{multline}
The non-infinitesimal exponent of $R$ above is zero since $\frac{1}{2} - \frac{1}{r_n} = \frac{1}{2n}$ and the definitions of $r_{n+1}$ and $a_n$ in \eqref{rnplusoneformula} and \eqref{andefn} imply that
 \[ \frac{1}{2} - \frac{1}{r_{n+1}} = \frac{ \frac{\alpha}{2n} - a_n }{\alpha}. \]
 Therefore, \eqref{tbs2} becomes
 \begin{multline*} \|Ef_{j,tan}\|_{BL^{r_m}_{{n-1},A/2}\left(Y \cap B_j \cap \mathcal{N}_{R^{1/2+ \delta_m}}(Z)\right)} \\
\lesssim  D^{O(1) } R^{(m-0.9)\epsilon} \gamma^{\frac{1}{r_m}-\frac{1}{p_{dec}} }
 R^{\delta(\log \overline{A} - \log A )} \|f_{j, tan}\|_2.  \end{multline*}
Substituting into \eqref{jsumtan} proves the last case in the tangential sub-case.

Now consider the transverse sub-case, where \eqref{jsumtr} holds. For fixed $j$, to bound
\[  \|Ef_{j,tr}\|_{BL^{r}_{{n-1},A/2}\left(Y \cap B_j \cap \mathcal{N}_{R^{1/2+ \delta_m}}(Z)\right)}^r, \] 
it will suffice to bound
\[  \|Ef_{j,tr}\|_{BL^{r}_{{n-1},A/2}\left(Y \cap B_j \cap \mathcal{N}_{R^{1/2+ \delta_m}}(Z(P_1, \dotsc, P_{n+1-m}))\right)}^r. \] 
Let $\mathcal{B}$ be a boundedly overlapping cover of $B_j \cap \mathcal{N}_{R^{1/2+ \delta_m}}(Z(P_1, \dotsc, P_{n+1-m}))$ by balls of radius $R^{1/2+\delta_m}$. By pigeonholing there is a fixed dyadic number $s$ and a subset $\mathcal{B}_s \subseteq \mathcal{B}$ such that 
\begin{multline*} \|Ef_{j,tr}\|_{BL^{r}_{{n-1},A/2}\left(Y \cap B_j \cap \mathcal{N}_{R^{1/2+ \delta_m}}(Z(P_1, \dotsc, P_{n+1-m}))\right)}^r \\
\lesssim (\log R) \|Ef_{j,tr}\|_{BL^{r}_{{n-1},A/2}\left(Y \cap B_j \cap \mathcal{N}_{R^{1/2+ \delta_m}}(Z(P_1, \dotsc, P_{n+1-m})) \cap \bigcup_{B \in \mathcal{B}_s} B )\right)}^r, \end{multline*}
and such that each ball $B$ from $\mathcal{B}_s$ has 
\[ \left\lvert B \cap \mathcal{N}_{\rho^{1/2+ \delta_m}}(Z(P_1, \dotsc, P_{n+1-m})\right\rvert \sim 2^s. \]
Let $N$ be the smallest integer greater than $\left\lvert B_{n+1}(0,R^{1/2+\delta_m})\right\rvert$. Pick a random set of $M \sim  R^{\delta} N/ 2^s$ vectors $b_1, \dotsc, b_M$ from $B_{n+1}(0, 100R^{1/2+\delta_m})$. It will be shown that that with high probability, every $B \in \mathcal{B}_s$ has the property that the union
\[ \bigcup_{l=1}^M  B \cap \mathcal{N}_{\rho^{1/2+\delta_m} }(Z(P_1, \dotsc, P_{n+1-m}))+b_{l} \]
contains $B \cap \mathcal{N}_{R^{1/2+\delta_m} }(Z(P_1, \dotsc, P_{n+1-m}))$, and the property that no radius 1 ball intersecting 
\[ 50B \cap \mathcal{N}_{R^{1/2+\delta_m} }(Z(P_1, \dotsc, P_{n+1-m})) \]
 intersects $\geq R^{2\delta}$ many of the sets $50B \cap \mathcal{N}_{\rho^{1/2+\delta_m} }(Z(P_1, \dotsc, P_{n+1-m}))+b_{l}$. To see the claim, for any $1$-ball $Q$ intersecting $50B$, for any fixed $l$, the probability that $Q$ intersects 
\[ 50B \cap \mathcal{N}_{\rho^{1/2+\delta_m} }((Z(P_1, \dotsc, P_{n+1-m}))+b_{l} \]
 is $\sim 2^s/N$. View the event that $Q$ intersects
\[ 50B \cap \mathcal{N}_{\rho^{1/2+\delta_m} }(Z(P_1, \dotsc, P_{n+1-m}))+b_{l} \]
 as a Bernoulli random variable with probability $p_l \sim  2^s/N$. The number of such sets intersecting $Q$ is a sum $S$ of $M \sim R^{\delta}N/2^s$ many independently distributed Bernoulli random variables $X_l$, where $X_l$ takes the value 1 with probability $p_l \sim 2^s/N$. For a positive integer $L$, the probability that $S$ exceeds $L$ is 
\[ \mathbb{P}(S \geq L) \leq \mathbb{P}\left( \exp\left(\sum_{i=1}^M X_i \right) \geq \exp(L) \right) \leq e^{-L} \left( \left(1- \frac{C2^s}{N}\right) + \frac{C2^s e}{N}  \right)^M, \]
for some large positive constant $C$. Hence, using $\log(1+x) \leq x$ for $x \geq 0$,
\[ \log \mathbb{P}(S \geq L) \leq -L + M(e-1) \frac{C2^s}{N}.  \]
Taking $L \sim R^{2\delta}$ and recalling that $M \sim R^{\delta}N/2^s$ gives 
\[ \log \mathbb{P}(S \geq L) \leq -(1/2) R^{2\delta}, \]
and hence the probability that more than $R^{2\delta}$ many sets intersect a given $Q$ is exponentially small in $R$. Similarly,  the probability that no sets 
\[ B \cap \mathcal{N}_{\rho^{1/2+\delta_m} }(Z(P_1, \dotsc, P_{n+1-m}))+b_{l} \]
intersect $Q$ is bounded by
\[ M\log\left( 1- \frac{c2^s}{N} \right), \]
for some small positive constant $c$, so using $\log(1-x) \leq -x/2$ for $0 \leq x \leq 1/2$ gives that the probability that no such sets intersect $Q$ is also exponentially small in $R$. Therefore, there are vectors $b_1, \dotsc, b_M$, such that for any 1-ball $Q$ and any ball $B \in \mathcal{B}_s$, no more than $R^{2\delta}$ of the sets $50B \cap \mathcal{N}_{\rho^{1/2+\delta_m} } (Z(P_1, \dotsc, P_{n+1-m}))+b_{l}$ intersect $Q$, and such that every 1-ball $Q$ intersecting $B$ also intersects
\[ \bigcup_l B \cap \mathcal{N}_{\rho^{1/2+\delta_m} }(Z(P_1, \dotsc, P_{n+1-m}))+b_{l}. \]
It follows (using the uncertainty principle) that 
\begin{multline} \label{lsum} \|Ef_{j,tr}\|_{BL^{r}_{{n-1},A/2}\left(Y \cap B_j \cap \mathcal{N}_{R^{1/2+ \delta_m}}(Z(P_1, \dotsc, P_{n+1-m}))\right)}^r \\
\lesssim R^{O(\delta)} \sum_b \|Ef_{j,tr}\|_{BL^{r}_{{n-1},A/2}\left(Y \cap B_j \cap \left[\mathcal{N}_{\rho^{1/2+ \delta_m}}(Z(P_1, \dotsc, P_{n+1-m}))+b \right]\right)}^r\\
\lesssim R^{O(\delta)} \sum_b \|Ef_{j,tr,b}\|_{BL^{r}_{{n-1},A/2}\left(Y \cap B_j \cap \left[\mathcal{N}_{\rho^{1/2+ \delta_m}}(Z(P_1, \dotsc, P_{n+1-m}))+b \right]\right)}^r. \end{multline} 
In the last line, $g_b$ means the wave packet decomposition of $g$ at the new scale $\rho$ instead of $R$, adapted to $B_j$, and then restricted to the wave packets which are tangent to $Z(P_1, \dotsc, P_{n+1-m})+b$ at the new scale. The reason that this restriction is valid is that for each $b$, each wave packet $f_{\theta, \nu}$ of $f_{j,tr }$ will be a sum of tangential wave packets in the new scale after it is re-decomposed;  this step is identical to \cite[\S 5.2.2]{ouwang}, which has more details. 

 \begin{sloppypar} By induction on the radius, for each $b$,
\begin{multline*} \|Ef_{j,tr,b}\|_{BL^{r}_{{n-1},A/4}\left(Y \cap B_j \cap \left[\mathcal{N}_{\rho^{1/2+ \delta_m}}(Z(P_1, \dotsc, P_{n+1-m}))+b \right]\right)}^r \\ \lesssim \left[\gamma^{b_m} \rho^{m \epsilon} \rho^{\delta\left( \log \overline{A} - \log (A/4) \right)} \rho^{-a_m} \right]^r\|f_{j,tr,b}\|_2^r. \end{multline*}  
Substituting this into \eqref{lsum} gives
\begin{multline} \label{pause} \|Ef_{j,tr}\|_{BL^{r}_{{n-1},A/2}\left(Y \cap B_j \cap \mathcal{N}_{R^{1/2+ \delta_m}}(Z(P_1, \dotsc, P_{n+1-m}))\right)}^r \lesssim R^{O(\delta)} \\
\left[\gamma^{b_m} \rho^{m \epsilon} \rho^{\delta\left( \log \overline{A} - \log A \right)} \rho^{-a_m} \right]^r  \left(\sup_b \|f_{j,tr,b}\|_2^{r-2}\right) \left( \sum_b  \|f_{j,tr,b}\|_2^2 \right). \end{multline} 
By the $\lesssim R^{O(\delta)}$-overlapping property of the sets  
\[ 50B \cap \mathcal{N}_{\rho^{1/2+\delta_m}}(Z(P_1, \dotsc, P_{n+1-m}))+b \]
for each $B \in \mathcal{B}$, and by following the same argument as in \cite[p.~3586]{ouwang},
\begin{equation} \label{borthogonality} \sum_b \|f_{j,tr,b}\|_2^2 \lesssim R^{O(\delta) } \|f_{j,tr}\|_2^2. \end{equation} 
The only difference compared to \cite{ouwang} is that the $R^{O(\delta)}$ overlapping causes a loss of $R^{O(\delta)}$ compared to the boundedly overlapping case. 
It follows from \cite[Lemma~5.13]{ouwang}, that
\begin{equation} \label{equidistribution} \sup_b \|f_{j, tr, b} \|_2^2 \lesssim R^{O(\delta_m)} \left( \frac{ \rho^{1/2}}{ R^{1/2}} \right)^{n+1-m} \|f_{j, tr}\|_2^2. \end{equation}
Strictly speaking, the above only holds with $f_{j,tr}$ replaced by $f_{j,tr}^{ess}$ in the left-hand side, which equals $f_{j,tr}$ plus a function which makes negligible contribution to the broad norms above, so by the same argument as in \cite{ouwang} it may be assumed that $f_{j,tr} = f_{j,tr}^{ess}$. \end{sloppypar}

Substituting \eqref{borthogonality} and \eqref{equidistribution} into \eqref{pause} gives 
\begin{multline} \label{powerr} \|Ef_{j,tr}\|_{BL^{r}_{{n-1},A/2}\left(Y \cap B_j \cap \mathcal{N}_{R^{1/2+ \delta_m}}(Z(P_1, \dotsc, P_{n+1-m}))\right)}  \\ \lesssim R^{O(\delta_m)}\gamma^{b_m} \rho^{m \epsilon} \rho^{\delta\left( \log \overline{A} - \log A \right)} \rho^{-a_m} \left( \frac{ \rho}{R} \right)^{ \left(\frac{n+1-m}{2}\right)\left(\frac{1}{2}-\frac{1}{r}\right)}  \|f_{j,tr}\|_2 \\
= R^{O(\delta_m)}\gamma^{b_m} \rho^{m \epsilon} \rho^{\delta\left( \log \overline{A} - \log A \right)} \left(\frac{R}{\rho} \right)^{a_m -\left(\frac{n+1-m}{2}\right)\left(\frac{1}{2}-\frac{1}{r}\right)} R^{-a_m}  \|f_{j,tr}\|_2 \\
\leq R^{O(\delta_m)} \gamma^{b_m} \rho^{m \epsilon} \rho^{\delta\left( \log \overline{A} - \log A \right)}  R^{-a_m}  \|f_{j,tr}\|_2, \end{multline}
where the last line follows from $r = r_m+\delta$ and the definition of $a_m$ (note $a_m \leq \frac{n+1-m}{2} \left(\frac{1}{2} - \frac{1}{r_n} \right)$ holds even whem $m=n$ since $a_n$ is decreasing in $\alpha$ and equality holds at $\alpha=n/2$).  

Taking both sides of \eqref{powerr} to the power $r$, summing over $j$, then taking both sides to the power $1/r$, gives
\begin{multline*} \|Ef_{j,tr}\|_{BL^{r}_{{n-1},A/2}\left(Y \cap B_j \cap \mathcal{N}_{R^{1/2+ \delta_m}}(Z(P_1, \dotsc, P_{n+1-m}))\right)} \\
\lesssim R^{O(\delta_m)} (R/\rho)^{-m\epsilon}  \gamma^{b_m} R^{m \epsilon} R^{\delta\left( \log \overline{A} - \log A \right)}  R^{-a_m} \left( \sum_j \|f_{j,tr}\|_2^r \right)^{1/r}. \end{multline*}
Finally, $\left( \sum_j \|f_{j,tr}\|_2^r \right)^{1/r}$ is bounded by $\left( \sum_j \|f_{j,tr}\|_2^2 \right)^{1/2}$, which by \cite[Eq.~(5.10)]{ouwang} is bounded by $ \|f\|_2$. Since $R/\rho \geq R^{\delta_{m-1}}$ and $\delta_m \ll \epsilon \delta_{m-1}$, , the induction closes and this finishes the proof in the (remaining) transverse case. \end{proof}

\begin{proposition} \label{inductionestimate} Suppose that $n \geq 4$ and $(n+1)/2 < \alpha < n$. Then, for any $\epsilon >0$, there exists $\delta_0 >0$ such that for all $0 \leq \delta \leq \delta_0$, for all $\gamma \geq 1$,
\begin{equation} \label{inductioninequality} \|Ef\|_{L^{p}(Y)}  \leq C_{\epsilon}  \left( \frac{\gamma}{M}\right)^{\frac{1}{2} -\frac{1}{p}} R^{s + \epsilon} \|f\|_{L^2(A(1))}, 
\end{equation}
where $p = \frac{2(n-2)}{n-4}$ and
\[ s = \begin{cases} \frac{2\alpha-1}{4 n} & \frac{n+1}{2} < \alpha \leq   \frac{n^2-5}{2(n-2)} \\
\frac{\alpha(n^2-3n+1) - n^2+2n+1}{2 n(n-1)(n-3)} &  \frac{n^2-5}{2(n-2)} < \alpha < n-1 \\
\frac{ \alpha(2n-1)-(n+1)}{4n(n-1)} & n-1 \leq \alpha \leq  \frac{3n^2-3n+2}{3n-2} \\
\frac{ \alpha-2 }{2(n-2)} & \frac{3n^2-3n+2}{3n-2}  < \alpha < n. \end{cases}   \]
for all smooth $f$ supported in $B_n(0,1) \setminus B_n(0,1/2)$, for any $R \geq 1$ and any disjoint union $Y$ of $M$ many radius $K^2$ balls in $B_{n+1}(0,R)$ with the property that $\left\lVert Ef \right\rVert_{L^p(Q)}$ is constant up to a factor of 2 over $K^2$ balls $Q \subseteq Y$, $K= R^{\delta}$, where $Y$ has the property that for any $1 \leq r \leq R$ there are at most $\gamma r^{\alpha}$ many $K^2$ balls from $Y$ intersecting any $r$-ball.
\end{proposition}
\begin{proof}[Proof that Proposition~\ref{inductionestimate} implies Theorem~\ref{L2estimate}] Let $R \geq 1$ and let $X$ be a disjoint union of unit balls in $B_{n+1}(0,R)$ with the property that for any $1 \leq r \leq R$ there are at most $\gamma r^{\alpha}$ many unit balls in any $r$-ball. Let $\epsilon >0$, and let 
$p= \frac{2(n-2)}{n-4}$
. By pigeonholing, there is a set $Y \subseteq X$, which is a disjoint union of unit balls, such that $\|Ef\|_{L^{p}(Q)}$ is constant up to a factor of 2 as $Q \subseteq Y$ varies over $Y$, and such that 
\[ \|Ef\|_{L^{2}(X)} \lesssim \left( \log R\right)^{1/2} \|Ef\|_{L^{2}(Y)} \leq \left( \log R\right)^{1/2} M^{\frac{1}{2} - \frac{1}{p}} \|Ef\|_{L^{p}(Y)}, \]
where $M$ is the number of unit balls $Q \subseteq Y$. Applying Proposition~\ref{inductionestimate} gives 
\[  \|Ef\|_{L^{2}(X)} \leq C_{\epsilon}  R^{s + \epsilon} \gamma^{\frac{1}{2} - \frac{1}{p}}\|f\|_{L^2(A(1))},\]
where the $s$ from Proposition~\ref{inductionestimate} matches the $s$ in the statement of Theorem~\ref{L2estimate}. \end{proof}

The proof of Proposition \ref{inductionestimate} is similar to \cite[Lemma~3.3]{harris2018}, but the details are included for completeness. 

\begin{proof}[Proof of Proposition~\ref{inductionestimate}] It suffices to prove the Proposition with $\delta=\delta_0$, where $\delta_0$ is implicitly chosen sufficiently small to make the argument below work. Assume that \eqref{inductioninequality} is true for all $\widetilde{R} \leq R/ K$; \eqref{inductioninequality} will then be shown for $R$. Sort the $K^2$-balls $Q \subseteq Y$ into ``$(n-1)$-broad'' and ``$(n-2)$-narrow'' cases, where $Q$ is called $(n-1)$-broad if for any $(n-2)$-dimensional subspaces $V_1, \dotsc, V_A$, there exists a $K^{-1}$-cap $\tau \subseteq \Gamma^n$ with $d(G(\tau), V) > K^{-2}$ such 
\[ \|Ef\|_{L^{p}(Q) } \leq K^{100n} \|Ef_{\tau} \|_{L^{p}(Q)}, \]
where $A=\overline{A}$ is the large constant from Theorem~\ref{broadweighted} when $m=n+1$. Otherwise $Q$ is called $(n-2)$-narrow. Here $f = \sum_{\tau} f \psi_{\tau}$ is the decomposition of $f$ obtained by forming a boundedly overlapping cover $\{\tau\}$ of $B_n(0,1) \setminus B_n(0,1/2)$ by tubes of dimensions $K^{-1} \times K^{-1} \times K^{-1} \times 1$, where $\{\psi_{\tau} \}_{\tau}$ is a smooth partition of unity subordinate to this cover. Here $A = \overline{A}$ is the large constant from Theorem~\ref{broadweighted}.  

Assume first that at least half of the balls $Q\subseteq Y$ are broad, which is the broad case. Let $q=r_{n+1}$ from Theorem~\ref{broadweighted}. By pigeonholing, there is a set $Y' \subseteq Y$ consisting of a fraction $\gtrsim 1/(\log R)$ of the (broad) balls $Q$, such that 
\begin{equation} \label{dyadic} \inf_{V_1, \dotsc, V_A \in G(n-2,n+1) } \sup_{d(G(\tau), V_a) \geq K^{-2} \forall a } \int_{2Q} |Ef_{\tau} |^q \end{equation}
is constant over $Q \subseteq Y'$ up to a factor of 2. Then, by the dyadically constant property of $\|Ef\|_{L^p(Q)}$ as $Q$ varies over $Y$, and the definition of the broad case,
\begin{multline*} \|Ef\|_{L^{p}(Y)} \lesssim \\
(\log R)^{1/p} K^{O(1)} \left( \sum_{Q \subseteq Y'} \inf_{V_1, \dotsc, V_A \in G(n-2,n+1) } \sup_{d(G(\tau), V_a) \geq K^{-2} \forall a } \int_Q |Ef_{\tau} |^p \right)^{1/p}, \end{multline*}
with the natural interpretation when $p=\infty$; see \eqref{infinityversion}. Since each $Ef_{\tau}$ has Fourier transform supported in a ball around the origin of radius $\sim 1$, the uncertainty principle yields $\|Ef_{\tau}\|_{L^p(Q)} \lesssim K^{O(1)} \|Ef_{\tau}\|_{L^q(2Q)}$, apart from a rapidly decaying error term which can ignored for the purposes below. Therefore
\begin{multline*} \|Ef\|_{L^{p}(Y)} \\
 \lesssim K^{O(1)}\left( \sum_{Q \subseteq Y'} \inf_{V_1, \dotsc, V_A \in G(n-2,n+1) } \sup_{d(G(\tau), V_a) \geq K^{-2} \forall a } \left(\int_Q |Ef_{\tau} |^q\right)^{p/q} \right)^{1/p} \\
= K^{O(1)}\left( \sum_{Q \subseteq Y'}  \left( \inf_{V_1, \dotsc, V_A \in G(n-2,n+1) } \sup_{d(G(\tau), V_a) \geq K^{-2} \forall a }\int_Q |Ef_{\tau} |^q\right)^{p/q} \right)^{1/p} \\
\lesssim K^{O(1)} M^{\frac{1}{p}-\frac{1}{q}}  \left(\sum_{Q \subseteq Y'}  \inf_{V_1, \dotsc, V_A \in G(n-2,n+1) } \sup_{d(G(\tau), V_a) \geq K^{-2} \forall a }\int_Q |Ef_{\tau} |^q\right)^{1/q}, \end{multline*}
where the dyadically constant property of the quantity in \eqref{dyadic} over $Q \subseteq Y'$ was used to reverse Hölder's inequality in the last line. The outer parts of this inequality hold also when $p=\infty$ by a more direct argument. By the definition of broad norm, this yields 
\[ \|Ef\|_{L^{p}(Y)} \lesssim K^{O(1)}  M^{\frac{1}{p}-\frac{1}{q}}  \|Ef\|_{BL^{q}_{{n-1},A}(Y)}. \]
Applying Theorem~\ref{broadweighted} to the above gives 
\begin{equation} \label{broadapplication} \|Ef\|_{L^{p}(Y)} \lesssim C_{\epsilon} R^{\epsilon/2} K^{O(1)}  M^{\frac{1}{p}-\frac{1}{q}} \gamma^{\frac{1}{q} - \frac{1}{p}} \|f\|_2, \end{equation}
where the power of $\gamma$ from Theorem~\ref{broadweighted} is always at most $\frac{1}{q}-\frac{1}{p}$, since $p=p_{dec} > r_{n-1}$, and increasing the power of $\gamma$ is permissible since $\gamma \geq 1$ (unless $Y$ is empty). By writing $M^{\frac{1}{p}-\frac{1}{q}} = M^{\frac{1}{p}-\frac{1}{2}} M^{\frac{1}{2} - \frac{1}{q}}$, the condition on $Y$ gives 
\[ M^{\frac{1}{p}-\frac{1}{q}} \lesssim M^{\frac{1}{p}-\frac{1}{2}} \gamma^{\frac{1}{2} - \frac{1}{q}} R^{\alpha \left( \frac{1}{2} - \frac{1}{q}\right)}. \]
Substituting this into \eqref{broadapplication} gives 
\[  \|Ef\|_{L^{p}(Y)} \lesssim C_{\epsilon} R^{\epsilon/2} K^{O(1)}  M^{\frac{1}{p}-\frac{1}{2}} \gamma^{\frac{1}{2}-\frac{1}{p}}  R^{\alpha \left( \frac{1}{2} - \frac{1}{q}\right)}\|f\|_2. \] By definition of the exponent $q=r_{n+1}$ from Theorem~\ref{broadweighted}, 
\[ \alpha\left( \frac{1}{2} -\frac{1}{q} \right) =\begin{cases} \frac{2\alpha-1}{4 n} & \frac{n+1}{2} < \alpha \leq  \frac{n^2-5}{2(n-2)} \\
\frac{\alpha(n^2-3n+1) - n^2+2n+1}{2 n(n-1)(n-3)} & \frac{n^2-5}{2(n-2)} < \alpha < n-1 \\
\frac{ \alpha(2n-1)-(n+1)}{4n(n-1)} & n-1 \leq \alpha < n \\\end{cases} \]
which is no larger than the exponent in the proposition statement, since for $\alpha >\frac{3n^2-3n+2}{3n-2}$,
\[ \frac{ \alpha(2n-1)-(n+1)}{4n(n-1)} <  \frac{ \alpha-2 }{2(n-2)},\]
so this finishes the proof of the broad case. 

Now suppose that at least half of the $K^2$-balls $Q \subseteq Y$ are narrow.

For each $\tau$, using the wave packet decomposition from \cite[Proposition~3.1]{harris2018}, decompose $f_{\tau} = \sum_{\Box \in \mathbb{T}_{\tau}} f_{\Box}$, where each $\Box$ has dimensions $\sim RK^{\delta-1} \times \dotsb \times RK^{\delta-1} \times RK^{\delta-2} \times R$, with long direction normal to the part of the cone corresponding to $\tau$, and short direction in the radial direction of the part of the cone corresponding to $\tau$. These boxes form a boundedly overlapping cover of $\mathbb{R}^{n+1}$. Each $Ef_{\Box}$ is negligible outside $\Box$, and $Ef = \sum_{\tau} \sum_{\Box \in \mathbb{T}_{\tau}} Ef_{\Box} = \sum_{\Box} Ef_{\Box}$. Let $\widetilde{R} = RK^{-2}$, and let $\widetilde{K} = \widetilde{R}^{\delta}$. For each $\tau$, the $\Box$ will (approximately) become an $\widetilde{R}$-ball under a Lorentz rescaling. For each $\tau$,  cover $\mathbb{R}^{n+1}$ by ellipsoids $S$ of dimensions $\sim \widetilde{K}^2K \times \dotsb \times \widetilde{K}^2 K \times \widetilde{K}^2 \times \widetilde{K}^2 K^2$, with short axis in the radial direction of the part of the cone corresponding to $\tau$, and long axis normal to the part of the cone corresponding to $\tau$. These ellipsoids will become $\widetilde{K}^2$-balls under a Lorentz rescaling. For fixed $\tau$ and $\Box \in \mathbb{T}_{\tau}$, sort the sets $S$ with $S \cap \Box \neq \emptyset$ according to the dyadic value $\kappa$ of $\|Ef_{\Box}\|_{L^p(S)}$ and further according to the dyadic value $\eta$ of the number of $K^2$-balls $Q \subseteq Y$ such that $Q \cap S \neq \emptyset$. For each $Q$, by pigeonholing there are sets $Y_{\Box}$, each of which is a union over such sets $S$ with fixed dyadic $\eta$ and $\kappa$, such that 
\[ \|Ef\|_{L^p(Q)} \lesssim (\log R)^{O(1)} \left\lVert \sum_{\Box: \tau(\Box) \in \bigcup_{a=1}^A V_a} \chi_{Y_{\Box} } Ef_{\Box} \right\rVert_{L^p(Q)}, \]
where, from the definition of $(n-2)$-narrow, $\tau \in V_a$ that $G(\tau)$ has distance at most $K^{-2}$ from $V_a$, and the $V_a$ are $(n-2)$-dimensional subspaces depending on $Q$.

By pigeonholing the balls $Q$, it may be assumed that the dyadic values $\kappa$ and $\eta$ ensuring the above are constant over $Q \subseteq Y$ (this uses the assumption that  $\|Ef\|_{L^p(Q)}$ is dyadically constant over $Q \subseteq Y$). By further pigeonholing, there is a collection $\mathbb{B}$ of sets $\Box$ such that 
\begin{equation} \label{stareq}  \|Ef\|_{L^p(Q)} \lesssim (\log R)^{O(1)} \left\lVert \sum_{\Box \in \mathbb{B} : : \tau(\Box) \in \bigcup_{a=1}^A V_a} \chi_{Y_{\Box} } Ef_{\Box} \right\rVert_{L^p(Q)}, \end{equation}
over a fraction $\gtrsim (\log R)^{-O(1)}$ of the balls $Q \subseteq Y$, and such that $\|f_{\Box}\|_2$ is constant over $\Box \in \mathbb{B}$ up to a factor of 2. By further pigeonholing, it may be assumed that for each $\Box \in \mathbb{B}$, each $Y_{\Box}$ is a union of $\sim \widetilde{M}$ ellipsoids $S$. By pigeonholing one more time, there is a set $Y' \subseteq Y$ with a fraction $\gtrsim (\log R)^{-O(1)}$ of the balls from $Y$, such that each $Q \subseteq Y'$ intersects $\sim \mu$ many sets $Y_{\Box}$ with $\Box \in \mathbb{B}$, and \eqref{stareq} holds for every $Q \subseteq Y'$. 

For each $K^2$-ball $Q \subseteq Y'$, by the $(n-2)$-narrow decoupling theorem for the cone (see e.g.~\cite[Theorem~7.1]{gaoliumiaoxi} for an exact statement),
\begin{equation} \label{mupause} \|Ef\|_{L^p(Q)} \lesssim  K^{O(\delta)} \left( \sum_{\Box \in \mathbb{B}: Y_{\Box} \cap Q \neq \emptyset} \| \chi_{Y_{\Box} } Ef_{\Box}\|_{L^{p}(2Q)}^2 \right)^{1/2}. \end{equation}
For each $Q \subseteq Y'$, it has been assumed that there are $\sim \mu$ many $\Box \in \mathbb{B}$ with $Y_{\Box} \cap Q \neq \emptyset$, so by \eqref{mupause} and Hölder's inequality,
\[ \|Ef\|_{L^p(Q)} \lesssim  K^{O(\delta)} \mu^{\frac{1}{2} - \frac{1}{p} } \left( \sum_{\Box \in \mathbb{B}} \| \chi_{Y_{\Box} } Ef_{\Box}\|_{L^{p}(2Q)}^p \right)^{1/p}. \] 
For finite $p$, taking both sides to the power $p$ and summing over $Q \subseteq Y'$ gives 
\begin{equation} \label{Efbound} \|Ef\|_{L^p(Y)} \lesssim  K^{O(\delta)} \mu^{\frac{1}{2} - \frac{1}{p} }\left( \sum_{\Box \in \mathbb{B}} \|  Ef_{\Box}\|_{L^{p}(Y_{\Box})}^p \right)^{1/p},\end{equation}
with this holding also if $p=\infty$. Let 
\[  s = \begin{cases} \frac{2\alpha-1}{4 n} & \frac{n+1}{2} < \alpha \leq   \frac{n^2-5}{2(n-2)} \\
\frac{\alpha(n^2-3n+1) - n^2+2n+1}{2 n(n-1)(n-3)} & \frac{n^2-5}{2(n-2)} < \alpha < n-1 \\
\frac{ \alpha(2n-1)-(n+1)}{4n(n-1)} & n-1 \leq \alpha \leq  \frac{3n^2-3n+2}{3n-2} \\
\frac{ \alpha-2 }{2(n-2)} &  \frac{3n^2-3n+2}{3n-2} < \alpha < n. \end{cases} \]
The definition of $s$ implies that for all $(n+1)/2 < \alpha < n$, 
\begin{equation} \label{snarrow} s \geq \frac{\alpha-2}{2(n-2)}. \end{equation}
By piecewise linearity and continuity of $s$, it suffices to check \eqref{snarrow} for the first and third parts of the definition of $s$, in which case \eqref{snarrow} is a straightforward calculation.  By Lorentz rescaling and the induction hypothesis, 
\begin{equation} \label{fboxbound} \|  Ef_{\Box}\|_{L^{p}(Y_{\Box})} \lesssim  K^{\frac{n+1}{p} - \frac{n-1}{2}} \left( \frac{\gamma(\Box)}{\widetilde{M}} \right)^{\frac{1}{2} - \frac{1}{p} } \widetilde{R}^{s+\epsilon} \|f_{\Box}\|_2. \end{equation}
Here $\gamma(\Box)$ is the smallest constant such that for any $1 \leq r \leq \widetilde{R}$, the number of ellipsoids $S \in \mathbb{S}_{\Box}$ intersecting any fixed parallel ellipsoid of dimensions $r K \times \dotsb \times rK \times r \times rK^2$, is at most $\gamma(\Box)r^{\alpha}$. This will become a Katz-Tao ball condition at scale $\widetilde{R}$ as in the Proposition statement, under a Lorentz rescaling. Similarly to \cite[Eq.~(3.19)]{harris2018} or the Schrödinger case in \cite[Eq.~(3.24)]{duzhang}, 
\[ M \mu \lesssim (\log R)^{O(1)} |\mathbb{B} | \widetilde{M} \eta, \]
and similarly to \cite[Eq.~(3.20)]{harris2018} or \cite[Eq.~(3.25)]{duzhang},
\[ \gamma(\Box) \lesssim K^{1+\alpha+O(\delta)} \gamma\eta^{-1}. \] 
Hence, \eqref{fboxbound} becomes 
\[ \|  Ef_{\Box}\|_{L^{p}(Y_{\Box})} \lesssim K^{-2(s+\epsilon) + \frac{n+1}{p} - \frac{n-1}{2} +(1+\alpha)\left( \frac{1}{2}-\frac{1}{p}\right)  +O(\delta) } \left( \frac{\gamma |\mathbb{B} |}{M \mu} \right)^{\frac{1}{2} - \frac{1}{p} } R^{s+\epsilon} \|f_{\Box}\|_2.  \]
Combining with \eqref{Efbound} yields 
\begin{multline*} \|Ef\|_{L^p(Y)} \lesssim K^{- 2(s+\epsilon)+ \frac{n+1}{p} - \frac{n-1}{2}+(\alpha+1)\left(\frac{1}{2} - \frac{1}{p} \right)+ O(\delta) } \\
\times \left( \frac{\gamma |\mathbb{B} |}{M} \right)^{\frac{1}{2} - \frac{1}{p} } R^{s+\epsilon}\left( \sum_{\Box \in \mathbb{B}} \|  f_{\Box}\|_2^p \right)^{1/p}. \end{multline*}
Since $\|f_{\Box}\|_2$ is dyadically constant, and by orthogonality (see \cite[Proposition~3.1]{harris2018}), this simplifies to 
\[ \|Ef\|_{L^p(Y)} \lesssim  K^{-2\epsilon} K^{\frac{n+1}{p} - \frac{n-1}{2}- 2s+(\alpha+1)\left(\frac{1}{2} - \frac{1}{p} \right) + O(\delta) } \left( \frac{\gamma}{M} \right)^{\frac{1}{2} - \frac{1}{p} } R^{s+\epsilon}\|f\|_2. \]
To close the induction, by \eqref{snarrow} it suffices to check that
\[ \frac{n+1}{p} - \frac{n-1}{2}- \frac{\alpha-2}{n-2}  +(\alpha+1)\left(\frac{1}{2} - \frac{1}{p} \right)= 0. \]
 After some algebra, this follows from $\frac{1}{2} - \frac{1}{p} = \frac{1}{n-2}$. \end{proof}

\section{Proof of Proposition~\ref{counterexample}}

\begin{proof}[Proof of Proposition~\ref{counterexample}]  Assume that $n \geq 4$, and let $m$ be any integer with $0 < m <(n+1)/2$, to be chosen. Let $\kappa \in (0,1/2)$ be a constant to be chosen. Let $R \geq 1$ be large. Let $\Omega'$ be the projection down to $\mathbb{R}^{n-m}$ of the $c R^{-1}$ neighbourhood of the set $\Lambda_m = \{ (\xi, |\xi|) \in \Gamma^{n-m} :R^{\kappa}(\xi, |\xi|) \in \mathbb{Z}^{n+1-m} \}$, where $c$ is a sufficiently small positive constant. It will be shown that the $(n-m)$-dimensional Lebesgue measure of $\Omega'$ is 
\begin{equation} \label{measurebound} \mathcal{H}^{n-m}(\Omega') \gtrapprox R^{\kappa(n-m-1) - (n-m)}. \end{equation} The set $\Lambda_m$ has $\gtrapprox R^{\kappa(n-m-1)}$ many points, since it is a union of rescaled lattice points on $\sim R^{\kappa}$ many spheres in $\mathbb{R}^{n-m}$ of integer radius $N \sim R^{\kappa}$, and the number of lattice points on each such sphere of large integer radius $N$ is $\gtrapprox N^{n-m-2}$. When $n-m=2$ this is trivial. When $n-m \geq 4$, it is a standard result from analytic number theory (see e.g.~\cite[Theorem~20.2]{iwanieckowalski}),  which holds even when $N$ is the square root of a large integer. When $n-m  = 3$, for $N$ odd there are at least $6N$ lattice points on the sphere of integer radius $N$. This follows from Stieltjes' formula:
\[ r_3(N^2) = 6 P \prod_{k=1}^s \left( q_k^{a_k} + \frac{2(q_k^{a_k}-1)}{q_k-1} \right) , \]
where $r_3(n)$ denotes the number of ways of writing $n$ as a sum of three squares, $N = 2^a P Q$ is the prime factorisation of $N$, with $P$ a product of primes congruent to $1 \bmod 4$, and $Q= \prod_{k=1}^s q_{k}^{a_k}$ a product of powers of distinct primes $q_k$ congruent to $3 \bmod 4$; see \cite{olds, olds2, pall} for proofs and references therein. Since there are still $\sim R^{\kappa}$ many odd $N$ in the range $N \sim R^{\kappa}$, \eqref{measurebound} still holds when $n-m=3$. This proves \eqref{measurebound}. Note that $n-m \geq 2$, since it has been assumed that $m < (n+1)/2$ and $n \geq 4$.  If $n-m = 2$ then $(n,m)=(4,2)$, and if $n-m=3$ then $(n,m) \in\{ (4,1), (5,2), (6,3)\}$.

Let $g = \chi_{\Omega'}$, so that (for $c$ now chosen sufficiently small)
\begin{equation} \label{lowerdim} \left\lvert \int_{\mathbb{R}^{n-m}} e^{2 \pi i \left\langle (\xi, |\xi| ), (Rx',Rt) \right\rangle} \, g(\xi) \, d\xi \right\rvert \gtrsim \mathcal{H}^{n-m}(\Omega'), \end{equation}
for all $(x',t)$ in the set 
\[ U = \left[R^{\kappa-1} \mathbb{Z}^{n+1-m} + B_{n+1-m}(0, c R^{-1} )\right] \cap B_{n+1-m}(0,1), \]
which has $(n+1-m)$-dimensional Lebesgue measure $\mathcal{H}^{n+1-m}(U) \sim R^{-\kappa (n+1-m)}$. The lower bound in \eqref{lowerdim} comes from the observation that if $(\xi,|\xi|) \in \mathcal{N}_{cR^{-1}}(\Lambda_m)$ and $(x,t) \in U$, then $\dist(\langle (\xi, |\xi|), (Rx,Rt) \rangle, \mathbb{Z})  \lesssim c$ for small $c$, and thus the integrand in \eqref{lowerdim} is essentially 1 on $\Omega'$. 

Define $\Omega \subseteq B_n(0,10n)$ by 
\[ \Omega = [0,  c R^{-1/2}]^{m} \times \Omega'. \]
 Then by \eqref{measurebound},
\begin{equation} \label{omegameasure} \mathcal{H}^{n}(\Omega) \gtrapprox  R^{\kappa(n-m-1) -n+ \frac{m}{2}}. \end{equation}
Let $f = \chi_{\Omega}$. Similarly to \eqref{lowerdim},
\begin{equation} \label{higherdim}  \left\lvert \int_{\mathbb{R}^n} e^{2 \pi i \left\langle (\xi, |\xi| ), (Rx,Rt) \right\rangle} \, f(\xi) \,  d\xi \right\rvert \gtrsim \mathcal{H}^{n}(\Omega), \end{equation}
for all $(x,t)$ in the set \[ V = [0, c R^{-1/2}]^m \times U, \]
which has Lebesgue measure $R^{-\frac{m}{2} -\kappa (n+1-m)}$. To see the lower bound in \eqref{higherdim}, recall that $\dist(\langle (\xi, |\xi|), (Rx,Rt) \rangle, \mathbb{Z})  \ll 1$ when $\xi' \in \Omega'$ and $(x',t) \in V$. Therefore, if $\xi= (\xi'', \xi') \in \Omega$ and $(x,t) = (x'',x',t) \in V$, then 
\[ \langle (\xi,|\xi|), (Rx,Rt) \rangle = \langle ( \xi'', \xi', |\xi''+\xi'|), (Rx'',Rx',Rt) \rangle  \in \mathbb{Z} + O(c),\]
by the formula $\sqrt{1+x} = 1 + O(x)$ for small $x$. Therefore, the integrand in \eqref{higherdim} is essentially 1 on $\Omega$, and the inequality in \eqref{higherdim} follows.

If $\nu = \chi_V$ and \[ m < \frac{n+1}{2}  < \alpha \leq  n+1-m, \]
to find $c_{\alpha}(\nu)$ requires considering $\nu(B(\cdot, r))$ for four ranges of $r$. Let $x \in \mathbb{R}^{n+1}$. If $0 < r < R^{-1}$, then
\[ \nu(B(x,r)) \lesssim  r^{n+1} \lesssim r^{\alpha} r^{n+1-\alpha} \lesssim r^{\alpha} R^{\alpha-n-1}. \]
If $R^{-1} \leq r \leq R^{\kappa-1}$, then (since 
$\alpha >m$)
\[ \nu(B(x,r)) \lesssim R^{-(n+1-m)}  r^m \lesssim r^{\alpha} r^{m-\alpha} R^{-(n+1-m)} \lesssim r^{\alpha} R^{\alpha-n-1}. \]
If $R^{\kappa-1} \leq r \leq R^{-1/2}$, then 
\begin{multline} \label{thirdcase}  \nu(B(x,r)) \lesssim r^m \left( \frac{r}{R^{\kappa-1} } \right)^{n+1-m} R^{-(n+1-m)} \\
= R^{-\kappa (n+1-m)} r^{\alpha} r^{n+1-\alpha} \lesssim r^{\alpha} R^{\frac{\alpha-n-1}{2} - \kappa (n+1-m) }.    \end{multline}
If $R^{-1/2} \leq r \leq 1$, then (since $\alpha \leq n+1-m$)
\begin{equation} \label{fourthcase} \nu(B(x,r)) \leq R^{-m/2} \left( \frac{r}{R^{\kappa-1} } \right)^{n+1-m} R^{-(n+1-m)} \lesssim r^{\alpha} R^{-\kappa (n+1-m) - \frac{m}{2} }.\end{equation}
Since $\alpha \leq n+1-m$, the fourth term is always larger than the third one. Hence, 
\[ c_{\alpha}(\nu) \lesssim  \max\left\{ R^{\alpha-n-1},  R^{-\kappa (n+1-m) - \frac{m}{2}} \right\}. \]
Therefore, choose
\begin{equation} \label{kappa1choice} \kappa = \frac{n + 1-  \frac{m}{2} - \alpha}{n+1-m} =: \kappa_1, \end{equation} which satisfies $0 < \kappa < 1/2$ since $(n+1)/2< \alpha \leq n+1-m$. This gives 
\begin{equation} \label{dimcdn} c_{\alpha}(\nu ) \sim \nu(\mathbb{R}^{n+1}) \sim R^{\alpha-n-1}. \end{equation} By \eqref{higherdim}, duality, and the definition of $\beta(\alpha, \Gamma^n)$, 
\begin{multline*} \mathcal{H}^{n}(\Omega) \nu(\mathbb{R}^{n+1}) \lessapprox \int \left\lvert \int_{\mathbb{R}^n} e^{2 \pi i \left\langle (\xi, |\xi| ), (Rx,Rt) \right\rangle} \, f(\xi) \,  d\xi \right\rvert \, d\nu(x,t) \\ \lesssim \left(\nu(\mathbb{R}^{n+1}) c_{\alpha}(\nu)\right)^{1/2} R^{-\beta(\alpha, \Gamma^n)/2}\mathcal{H}^{n}(\Omega)^{1/2}.\end{multline*}
By \eqref{dimcdn}, this simplifies to
\[ R^{\beta(\alpha, \Gamma^n)} \lessapprox \mathcal{H}^n(\Omega)^{-1}. \]
Using \eqref{omegameasure}
  and letting $R \to \infty$ gives 
 \[ \beta(\alpha, \Gamma^n) \leq n - \frac{m}{2}  - \kappa(n-m-1). \]
Since $\kappa = \frac{n + 1-  \frac{m}{2} - \alpha}{n+1-m}$, this simplifies to 
\[ \beta(\alpha, \Gamma^n) \leq \alpha -1 + 2\kappa. \]
Substituting $\kappa = \frac{n + 1-  \frac{m}{2} - \alpha}{n+1-m}$ gives
\[ \beta(\alpha, \Gamma^n) \leq \alpha -1 + \frac{2n + 2-m-2\alpha}{n+1-m}. \]
This can be written as
\begin{equation} \label{caseone} \beta(\alpha, \Gamma^n) \leq \alpha  - \frac{2\alpha -n - 1}{n+1-m}. \end{equation}
In particular, if $\alpha > (n+1)/2$, the bound is improved by taking $m$ as large as possible subject to the restriction $\alpha \leq n+1-m$, or equivalently $m \leq n+1-\alpha$. Since it has been assumed that $\alpha > (n+1)/2$, for such $m$ the requirement $\alpha >m$ is redundant. Therefore, the optimal choice is to take $m$ such that $\alpha \in (n-m, n+1-m]$. 

\subsection{Alternative case} 
If $\alpha \geq n+1-m$, then the adjusted version of \eqref{fourthcase} will match \eqref{thirdcase}. Hence
\[ c_{\alpha}(\nu) \lesssim \max\left\{ R^{\alpha-n-1},  R^{-\kappa (n+1-m) + \frac{\alpha-n-1}{2}}\right\}, \]
so let
\begin{equation} \label{kappa2choice}  \kappa = \frac{n+1-\alpha}{2(n+1-m)} =:\kappa_2\end{equation}
(which is in $(0,1/2)$ since $\alpha >m$), to get 
\begin{equation} \label{cvalue}  c_{\alpha}(\nu) \lesssim R^{\alpha-n-1}. \end{equation}
which is now strictly larger than
\begin{equation} \label{muvalue} \nu(\mathbb{R}^{n+1}) \sim R^{-\kappa(n+1-m) - \frac{m}{2}}. \end{equation}
Following the rest of the proof above,  
\[ R^{\beta(\alpha, \Gamma^n)} \lessapprox \mathcal{H}^n(\Omega)^{-1} \nu(\mathbb{R}^n)^{-1} c_{\alpha}(\nu). \]
By substituting \eqref{cvalue} and \eqref{muvalue}, and the value $\mathcal{H}^n(\Omega) \gtrapprox R^{\kappa(n-m-1)-n + \frac{m}{2}}$ from \eqref{omegameasure}, and letting $R \to \infty$, 
\[ \beta(\alpha, \Gamma^n) \leq \alpha-1 + 2\kappa. \]
Substituting the value of $\kappa$ from \eqref{kappa2choice} gives
\begin{equation} \label{casetwo} \beta(\alpha, \Gamma^n) \leq \alpha-1 + \frac{n+1-\alpha}{n+1-m}.  \end{equation}
This is an increasing function of $m$, so the optimal choice is to take $m$ as small as possible subject to the restriction $\alpha \geq n+1-m$, or $m \geq n+1-\alpha$.  Therefore, the optimal choice is to take $m$ such that $\alpha \in [n+1-m, n+2-m)$. 

\subsection{Optimising both bounds}

If $(n+1)/2 < \alpha < n$, choose a unique integer $m$ such that $n-m \leq \alpha < n+1-m$. Since $\alpha < n$ it follows that $m \geq 1$ from the lower bound, and since $\alpha > (n+1)/2$ it follows that $m < (n+1)/2$ from the upper bound. If $m < (n-1)/2$ then by \eqref{caseone} and \eqref{casetwo},
\[ \beta(\alpha, \Gamma^n) \leq \min\left\{\alpha  - \frac{2\alpha -n - 1}{n+1-m}, \alpha-1 + \frac{n+1-\alpha}{n-m} \right\}.\]
The two terms in the minimum are equal when 
\[ \alpha = n+1-m - \frac{m}{n-m-1}. \]
The first bound is better for larger $\alpha$, and the second is better for smaller $\alpha$.
Thus,
\[ \beta(\alpha, \Gamma^n) \leq \begin{cases} \alpha-1 + \frac{n+1-\alpha}{n-m} & n-m \leq \alpha \leq n+1-m - \frac{m}{n-m-1} \\
\alpha  - \frac{2\alpha -n - 1}{n+1-m} & n+1-m - \frac{m}{n-m-1} \leq \alpha < n+1-m. \end{cases}  \]
Finally, if $(n-1)/2 \leq m < (n+1)/2$, then either $n$ is odd and $m = (n-1)/2$, or $n$ is even and $m = n/2$. In either case only the first term in the minimum may be used:
\[ \beta(\alpha, \Gamma^n) \leq \alpha  - \frac{2\alpha -n - 1}{n+1-m}. \]
This finishes the proof.  \end{proof}

\bibliographystyle{plainurl}
\bibliography{divergence}

\end{document}